\documentclass[12pt,a4j,twoside,reqno]{amsart}

\usepackage[top=28truemm,bottom=28truemm,left=26truemm,right=26truemm]{geometry}

\usepackage{amssymb, amsmath, amsthm}
\usepackage{amsfonts}
\usepackage{color}
\usepackage{cancel}
\usepackage{mathrsfs}
\usepackage{yfonts}
\usepackage{accents}
\usepackage{relsize}
\usepackage{graphicx}
\usepackage[utf8]{inputenc}
\usepackage{cite}
\usepackage{multicol}
\usepackage{multirow}
\usepackage{float}
\usepackage{relsize}
\usepackage{array}
\usepackage{adjustbox}
\usepackage{tikz}
\usepackage{comment}
\usepackage[abbrev]{amsrefs}
\usepackage{mathtools}

\numberwithin{equation}{section}

\numberwithin{equation}{section}

\renewcommand{\d}{\mathrm{d}}

\newcommand{\E}{\mathcal{E}}
\newcommand{\F}{\mathcal{F}}
\renewcommand{\H}{\mathcal{H}}

\newcommand{\N}{\mathbb{N}}

\newcommand{\R}{\mathbb{R}}

\newcommand{\Tr}{\mathrm{Tr}}

\newcommand{\dv}{{\rm{div}}}

\newcommand{\loc}{\mathrm{loc}}

\newcommand{\dom}{{\partial\Omega}}
\newcommand{\ddom}{{\partial_D \Omega}}
\newcommand{\ndom}{{\partial_N \Omega}}

\newcommand{\ind}{{I_{(-\infty,0]}}}

\newcommand{\dind}{{\partial \ind}}

\newcommand{\ztau}{{z_\tau}}
\newcommand{\ztaubar}{\overline{z}_\tau}
\newcommand{\ftau}{{f_\tau}}
\newcommand{\ftaubar}{\overline{f}_\tau}

\newcommand{\sigmatau}{{\sigma_\tau}}
\newcommand{\sigmataubar}{\overline{\sigma}_\tau}

\newcommand{\vep}{\varepsilon}
\newcommand{\vphi}{\varphi}

\newcommand{\included}{\hookrightarrow}
\newcommand{\wto}{\rightharpoonup}

\newcommand{\disp}{\displaystyle}

\newcommand{\mel}[1]{\MoveEqLeft{#1}}

\theoremstyle{plain}
\newtheorem{thm}{Theorem}[section]
\newtheorem{prop}[thm]{Proposition}
\newtheorem{lem}[thm]{Lemma}

\theoremstyle{definition}
\newtheorem{defi}[thm]{Definition}

\theoremstyle{remark}
\newtheorem{rem}[thm]{Remark}

\makeatletter
\@namedef{subjclassname@2020}{%
  \textup{2020} Mathematics Subject Classification}
\makeatother

\title{Quasistatic evolution equations with irreversibility arising from fracture mechanics}
\author{Kotaro Sato}
\address{Mathematical Institute, Tohoku University, Sendai 980-8578, Japan}
\email{kotaro.sato.t8@dc.tohoku.ac.jp}



\begin{document}

\maketitle

\begin{abstract}
	In this paper, the global-in-time $ L^2 $-solvability of the initial-boundary value problem for differential inclusions of doubly-nonlinear type is proved.
	This problem arises from fracture mechanics, and it is not covered by general existence theories due to the degeneracy and singularity of a dissipation potential along with the nonlinearity of elliptic terms.
	The existence of solutions is proved based on a minimizing movement scheme, which also plays a crucial role for deriving qualitative properties and asymptotic behaviors of strong solutions.
	Moreover, the solutions to the initial-boundary value problem comply with three properties intrinsic to brittle fracture: \textit{complete irreversibility}, \textit{unilateral equilibrium of an energy} and \textit{an energy balance law}, which cannot generally be realized in dissipative systems.
	Furthermore, long-time dynamics of strong solutions are revealed, i.e., each stationary limit of the global-in-time solutions is characterized as a solution to the stationary problem.
\end{abstract}

\section{Introduction}\label{sec:intro}

Let $ N = 1, 2, \dots $ and let $ \Omega \subset \R^N $ be a bounded domain with Lipschitz boundary $ \dom $, and $ T > 0 $.
Let $ \ddom, \ndom \subset \dom $ be relatively open subsets satisfying $ \ddom \cap \ndom = \emptyset $ and $ \H^{N-1}(\dom \setminus (\ddom \cup \ndom))=0 $, where $ \H^{N-1} $ denotes the $ (N-1) $-dimensional Hausdorff measure.
Moreover, let $ \partial_\nu $ denote the outward normal derivative on $ \ndom $.
The present paper is concerned with the following differential inclusion:
\begin{equation}\label{eq}
	\dind \left( \partial_t z \right) - \Delta z + \lambda z + \sigma \gamma(z) \ni f \quad \mbox{ in } \ \Omega \times (0,T),
\end{equation}
equipped with the initial and boundary conditions,
\begin{equation}\label{bcic}
	\left\lbrace
	\begin{alignedat}{4}
		&z = 0 & \quad &\mbox{ on } \ \ddom \times (0,T),\\
		&\partial_\nu z = 0 &&\mbox{ on } \ \ndom \times (0,T),\\
		&z(\cdot,0) = z_0(\cdot) &&\mbox{ in } \ \Omega,
	\end{alignedat}
	\right.
\end{equation}
where $ \lambda \geq 0 $ is a constant, $ z_0 = z_0(x) $, $ f=f(x,t) $ and $ \sigma = \sigma(x,t) \geq 0 $ are given functions and $ \gamma : \R \to \R $ is a continuous function.
Moreover, $\ind : \R \to [0,+\infty]$ stands for the indicator function supported over $(-\infty, 0]$, i.e., $\ind(s) = 0$ if $s \leq 0$ and $\ind(s) = +\infty$ if $s > 0$.
Furthermore, $\dind : \R \to 2^\R$ denotes the subdifferential of $\ind$, that is, 
\begin{align*}
	\dind \left( s \right) &= \left\{ \, \xi \in \R \, \colon \, \xi(r-s) \leq 0 \quad \mbox{for } \, r \in \R \, \right\}\\
	&= \begin{cases}
		\{0\} &\mbox{ if } \ s < 0,\\
		[0,+\infty) &\mbox{ if } \ s = 0,\\
		\emptyset &\mbox{ if } \ s>0
	\end{cases}
\end{align*}
for $ s \in \R $ with domain $D(\dind) = (-\infty,0]$.
The differential inclusion \eqref{eq} can be rewritten as a variational inequality of the form
\begin{equation}\label{vi}
	\left\lbrace
	\begin{alignedat}{4}
		&\partial_t z \leq 0, \quad -\Delta z + \lambda z + \sigma \gamma(z) \leq f & \quad &\mbox{ in } \ \Omega \times (0,T),\\
		&\partial_t z \left( -\Delta z + \lambda z + \sigma \gamma(z) - f \right) = 0 &&\mbox{ in } \ \Omega \times (0,T).
	\end{alignedat}
	\right.
\end{equation}
In particular, solutions to \eqref{eq} (and \eqref{bcic}) are nonincreasing in time.

On the other hand, let $ H := L^2(\Omega) $ and $ V := \{ v \in H^1(\Omega) \colon v = 0 \, \mbox{ on } \, \ddom \} $ (see also Subsection \ref{ssec:fs}), and set $ z(t) := z(\cdot,t) $, $ f(t) := f(\cdot,t) $ and $ \sigma(t) := \sigma(\cdot,t) $.
Then the equation \eqref{eq} equipped with the boundary condition \eqref{bcic} can be reduced into the Cauchy problem of the following doubly-nonlinear evolution equation:
\begin{equation}\label{eq:general}
	\partial \psi (\partial_t z(t)) + \partial \vphi^t (z(t)) \ni f(t) \quad \mbox{ in } \ H, \quad 0 < t < T,
\end{equation}
where $ \psi, \vphi^t : H \to \R \cup \{ +\infty \} $ are defined by
\begin{align*}
	\psi(v) &=
	\left\lbrace
		\begin{alignedat}{4}
			&0 &\quad &\mbox{ if } \ v \leq 0 \quad \mbox{ a.e.~in } \ \Omega,\\
			&{+\infty} &&\mbox{ otherwise,}
		\end{alignedat}
	\right.\\
	\vphi^t(v) &=
	\left\lbrace
		\begin{alignedat}{4}
			&\frac{1}{2} \int_\Omega |\nabla v|^2 \, \d x + \frac{\lambda}{2} \int_\Omega |v|^2 \, \d x + \int_\Omega \sigma(t) \hat{\gamma}(v) \, \d x &\quad &\mbox{ if } \ v \in V,\\
			&{+\infty} &&\mbox{ otherwise}
		\end{alignedat}
	\right.
\end{align*}
for $ v \in H $ and $ t \in (0,T) $, and $ \partial \psi, \partial \vphi^t : H \to 2^H $ are the subdifferential operators of $ \psi $ and $ \vphi^t $, respectively.
Here we denote by $ \hat{\gamma} $ the primitive function of $ \gamma $ which satisfies that $ \sigma(t) \hat{\gamma}(v) \in L^1(\Omega) $ for $ v \in V $ and a.e.~$ t \in (0,T) $.
For the case where $ \vphi^t $ is independent of $ t $, i.e., $ \vphi^t \equiv \vphi $, such doubly-nonlinear equations have been studied in \cite{arai, barbu1975, colli, cv, sss, e2ciec} for example, and in particular, some results concerning the existence of solutions can be found.
However, there seems no general theory covering the equation \eqref{eq:general} due to the non-coercivity and the singularity of $ \psi $ along with the nonlinearity of $ \partial \vphi $; indeed, in \cite{e2ciec}, equations of type \eqref{eq:general} are treated with $ \gamma \equiv 0 $.
However, the proof exhibited in \cite{e2ciec} depends on the linearity of $ \partial \vphi $, and therefore, the strategy mentioned in \cite{e2ciec} is inapplicable to our problems.
In particular, the equation treated in \cite{e2ciec} can be regarded as a specific case of \eqref{eq}.
Moreover, we also address the case where $ \vphi^t $ depends on $ t $, since the coefficient function $ \sigma $ in \eqref{eq} depends on $ t $.
This type of equation is derived from a doubly-nonlinear system \eqref{eq:system} in our study, which comes from a fracture model explained later (see also \eqref{eq:example} and \eqref{eq:single1} below).

The problem concerned in the present paper is motivated by some variational problems arising from mechanics of brittle fractures, such as the growth of cracks.
Classically, such problems are based on the Griffith criterion, which states that a crack in an elastic body propagates as soon as the (local) energy release rate becomes larger than the toughness of the body (see \cite{griffith}).
On the other hand, a quasistatic framework introduced by Francfort and Marigo (see \cite{fm}) is well known as a pioneering work approaching the crack growth problems from a variational perspective, which is based on a global minimization of the so-called \emph{Francfort-Marigo energy} with an irreversibility condition of fracture.
The Francfort-Marigo energy consists of the sum of the elastic energy and the surface energy, which is of the form
\[
	\F(u,\Gamma) = \frac{1}{2} \int_{\Omega \setminus \Gamma} \left| \nabla u \right|^2 \d x + k \H^{N-1}(\Gamma),
\]
where $ u : \Omega \to \R^m $ for $ m = 1, 2, \dots $ is a displacement field, and on the other hand $ \Gamma \subset \overline{\Omega} $ is a crack set on the elastic body $ \Omega $.
Moreover, $ \H^{N-1} $ denotes the $ (N-1) $-dimensional Hausdorff measure.
The constant $ k>0 $ represents the toughness of the material.
Given $ g : \ddom \setminus \Gamma \to \R^m $, let $ \tilde{\mathcal{F}}(\Gamma;g) $ be the minimum value of $ \mathcal{F}(u,\Gamma) $ with respect to $ u $ satisfying the non-homogeneous Dirichlet boundary condition $ u = g $ on $ \ddom \setminus \Gamma $.
In~\cite{fm}, for boundary displacement fields $g = g(x,t)$ varying in time, a \textit{quasistatic evolution} of crack $ t \mapsto \Gamma (t) $ on $[0,T]$ is defined to satisfy the following three conditions:
\begin{enumerate}
	\item $ \Gamma(t_1) \subset \Gamma(t_2) \, $ for all $ 0 \leq t_1 \leq t_2 \leq T $;
	\item for each $ t \in [0,T]$, it holds that $ \tilde{\mathcal{F}}(\Gamma(t);g(\cdot,t)) \leq \tilde{\mathcal{F}}(\Gamma;g(\cdot,t)) $ for any closed subsets $ \Gamma \subset \overline{\Omega}$ satisfying $ \bigcup_{s<t} \Gamma(s) \subset \Gamma $;
	\item the total energy function $ t \mapsto \tilde{\mathcal{F}}(\Gamma(t);g(\cdot,t)) $ is absolutely continuous on $[0,T]$ and its derivative is precisely obtained by the external loads, i.e., it holds that
	\[
		\frac{\d}{\d t} \, \tilde{\mathcal{F}}(\Gamma(t);g(\cdot,t)) = \int_{\Omega \setminus \Gamma(t)} \nabla u_*(\cdot,t) \cdot \nabla \dot{g}(\cdot,t) \, \d x \quad \mbox{ for a.e.~} t \in (0,T),
	\]
	where $ u_*=u_*(x,t) $ is the minimizer of $ \F $, i.e., $ \F(u_*(\cdot,t),\Gamma(t)) = \tilde{\mathcal{F}}(\Gamma(t);g(\cdot,t)) $ and $ u_* $ satisfies the boundary condition $ u_*(\cdot,t) = g(\cdot,t) $ on $ \ddom \setminus \Gamma $ for a.e.~$ t \in (0,T) $.
	Moreover, $ \dot{g} $ denotes the time derivative of $ g $.
\end{enumerate}
Condition (i) constrains the crack evolution to be completely irreversible, while (ii) can be regarded as a unilateral minimization of the energy, i.e., the crack set $ \Gamma(t) $ minimizes the Francfort-Marigo energy at each time among all cracks containing the previous ones.
Condition (iii) describes an energy balance criterion of the whole system, and in particular, the crack never evolves unless the external load $g(x,t)$ varies in time and the energy never dissipates spontaneously.
Some existence results on the quasistatic evolution $ t \mapsto \Gamma(t) $ have already been obtained, e.g., in~\cite{dt2002}, where anti-planar shear fracture in two-dimensional cylindrical elastic bodies is considered, and moreover, in~\cite{dft2005}, the former result is extended to more general configurations such as nonlinear elasticity for heterogeneous and anisotropic materials in general dimension (see also \cite{chambolle2003, fl} and references therein).

In view of numerical analysis, Ambrosio and Tortorelli \cite{at1990, at1992} introduced a regularized functional, which is called an \emph{Ambrosio-Tortorelli energy} $ \mathcal{F}_\vep $ (here AT functional for short) with an approximation parameter $ \vep > 0 $.
Their approach is to replace the crack set $ \Gamma \subset \overline{\Omega} $ with the phase field $ z = z(x,t) \in [0,1] $, which takes $ 1 $ on the intact part and $ 0 $ on the cracked part, of the material, respectively.
In particular, in terms of the phase field $ z = z(x,t) $ and the displacement field $ u = u(x,t) $, AT functional $ \mathcal{F}_\vep $ is written as
\[
	\mathcal{F}_\vep(u,z;g) = \frac{1}{2} \int_\Omega \left( z^2 + \delta_\vep \right) \left| \nabla u \right|^2 \d x + \frac{\vep}{2} \int_\Omega \left| \nabla z \right|^2 \d x + \frac{1}{2\vep} \int_\Omega \left( 1-z \right)^2 \d x,
\]
where $(u,z) \in H^1(\Omega) \times H^1(\Omega)$ satisfies that $ z = 1 $, $u = g$ on $ \ddom $ for the boundary displacement field $ g $ and $ 0 \leq z \leq 1 $ a.e.~in $\Omega$.
Moreover, $0 < \delta_\vep \ll \vep$ is a constant for regularization.
For time-varying boundary fields $ g = g(x,t) $, minimizing orbit $ t \mapsto (u_\vep(t), z_\vep(t)) $ for AT functional is similarly defined as (i)--(iii) above, which is also called quasistatic evolution for AT functional.
However, in this model, the monotone growth condition (i) is replaced by the non-increasing condition of $ t \mapsto z_\vep(\cdot,t) $.
Furthermore, it is known that $ \mathcal{F}_\vep $ converges to the original Francfort-Marigo energy $ \mathcal{F} $ as $ \vep \to 0_+ $ in the sense of $ \Gamma $-convergence (see \cite{at1990}).

While the quasistatic evolution for AT functional is constructed in \cite{giacomini} via the minimizing movement scheme, there are less analysis for such evolutions from PDE points of view, and in particular, such quasistatic evolutions have not yet been characterized as solutions to the corresponding Euler-Lagrange equation.
Here, a formal Euler-Lagrange system of $ \mathcal{F}_\vep $ with completely irreversible constraint can be seen as follows:
\begin{equation}\label{eq:system}
	\left\{
	\begin{alignedat}{4}
		&-\dv \left( (z^2 + \delta_\vep) \nabla u \right) = h &\quad &\mbox{ in } \, \Omega \times (0,T),\\
		&\disp \dind \left( \partial_t z \right) - \vep \Delta z + \frac{z-1}{\vep} + z \left| \nabla u \right|^2 \ni 0 &\quad &\mbox{ in } \, \Omega \times (0,T),
	\end{alignedat}
	\right.
\end{equation}
where $ h = h(x,t) $ is a given function with null average in $ x $, and moreover, we assign the initial-boundary condition
\begin{equation}\label{bcic:system}
 	\left\lbrace
		\begin{alignedat}{4}
			&z = u = 0 &\quad& \mbox{ on } \, \ddom \times (0,T),\\
			&\partial_\nu z = \partial_\nu u = 0 &\quad& \mbox{ on } \, \ndom \times (0,T),\\
			&z(\cdot,0) = z_0 &\quad& \mbox{ in } \Omega.
		\end{alignedat}
	\right.
\end{equation}
Here we find that the subdifferential operator $ \dind $ plays the role of a Lagrange multiplier to constrain the phase field $ z $ not to increase in time.
Although it seems to be natural to tackle with the problem \eqref{eq:system} and \eqref{bcic:system} for further analysis of brittle fractures, there seems no complete result proving the well-posedness of this system, and in particular, the existence of solutions has not yet been proved as far as the author knows.

In the present paper, let us restrict ourselves into the case of $ 1 $-dimensional space $ I = (-1,1) $, and moreover, for simplicity, we impose the (homogeneous) Dirichlet boundary condition on $ z $ and the (homogeneous) Neumann boundary condition on $ u $, respectively.
Then \eqref{eq:system} (and \eqref{bcic:system}) reads
\begin{equation}\label{eq:system1}
	\left\{
	\begin{alignedat}{4}
		&-\partial_x \left( (z^2 + \delta_\vep) u_x \right) = h &\quad &\mbox{ in } \ I \times (0,T),\\
		&\disp \dind \left( \partial_t z \right) - \vep z_{xx} + \frac{z-1}{\vep} + z \left| u_x \right|^2 \ni 0 &&\mbox{ in } \ I \times (0,T)
	\end{alignedat}
	\right.
\end{equation}
and
\begin{equation}\label{bcic:system1}
 	\left\lbrace
		\begin{alignedat}{4}
			&z = 0 &\quad &\mbox{ on } \ \{ -1,1 \} \times (0,T),\\
			&u_x = 0 &&\mbox{ on } \ \{ -1,1 \} \times (0,T),\\
			&z(\cdot,0) = z_0 &&\mbox{ in } \ I,
		\end{alignedat}
	\right.
\end{equation}
where $ u_x = \partial u / \partial x $ and $ z_{xx} = \partial^2 z / \partial x^2 $.
Then, it is noteworthy that the system \eqref{eq:system1} can be (formally) reduced into a single equation of $ z $ combined with the boundary condition \eqref{bcic:system1}.
Indeed, integrating the first equation of \eqref{eq:system1} over $ (-1,x) $ for each $ x \in I $, we have
\begin{align*}
	\int^x_{-1} h(y,t) \, \d y &= - \left( z(x,t)^2 + \delta_\vep \right) u_x(x,t), \quad t \in (0,T),
\end{align*}
and hence it follows that
\begin{equation}\label{eq:ux}
	\left| u_x(x,t) \right|^2 = \frac{1}{\left| z(x,t)^2 + \delta_\vep \right|^2} \left( \int^x_{-1} h(y,t) \, \d y \right)^2, \quad (x,t) \in I \times (0,T).
\end{equation}
Substituting \eqref{eq:ux} into the second equation of \eqref{eq:system1}, we obtain
\begin{equation}\label{eq:sin}
	\dind \left( \partial_t z \right) - \vep z_{xx} + \frac{z-1}{\vep} + \left( \int^x_{-1} h(y,t) \, \d y \right)^2 \frac{z}{\left| z(x,t)^2 + \delta_\vep \right|^2} \ni 0.
\end{equation}
Therefore, by setting $ \lambda := 1 / \vep^2 $, $ f(x,t) := \lambda = 1 / \vep^2 $ and
\begin{equation}\label{eq:example}
	\left\{
	\begin{alignedat}{4}
		\sigma(x,t) &:= \left( \int^x_{-1} h(y,t) \, \d y \right)^2 & \quad &\mbox{ for } \, (x,t) \in I \times (0,T),\\
		\gamma(s) &:= \frac{s}{\vep \left( s^2 + \delta_\vep \right)^2} &&\mbox{ for } \, s \in \R,
	\end{alignedat}
	\right.
\end{equation}
we can rewrite the equation \eqref{eq:sin} as
\begin{equation}\label{eq:single1}
	\dind \left( \partial_t z \right) - z_{xx} + \lambda z + \sigma \gamma (z) \ni f \quad \mbox{ in } \ I \times (0,T).
\end{equation}
Here we used the fact that $ \dind(\partial_t z) / \vep = \dind(\partial_t z) $.
The main equation \eqref{eq} is an $ N $-dimensional version of \eqref{eq:single1}, and this is the reason why we choose the nonlinearity of the form $ \sigma \gamma(z) $ in \eqref{eq}.
Moreover, by virtue of such a formulation from PDE points of view, we can discuss the long-time dynamics of global-in-time solutions to \eqref{eq} (see Theorem \ref{thm:asympt}).
In particular, in 1-dimensional setting, we can identify the stationary limits of global-in-time solutions to the system \eqref{eq:system1} equipped with the initial-boundary condition \eqref{bcic:system1}.

On the other hand, the equation \eqref{eq} is classified as a \textit{rate-independent} equation by virtue of the homogeneity of the dissipation potential $ \ind $ of degree $ 1 $.
However, it is known that, in general, solutions of rate-independent systems combined with nonconvex energies are not continuous in time, i.e., they may have jump points in time even in a local formulation (see also \cite{kmz, mz}).
Indeed, rate-independent models are originally introduced to be applied to some variational problems which may include the discontinuity in time, such as fracture problems.
However, this paper is concerned with a convex energy functional (see (i) of Theorem \ref{thm:main}), so that the solutions are continuous in time and we can focus on the difficulty which comes from the dissipation potential and the nonlinearity of $ \partial \vphi^t $.
Furthermore, we consider the strong ($ L^2 $-)solutions of \eqref{eq} and \eqref{bcic}.

Let $ X = \{ v \in H^2(\Omega) \colon \partial_\nu v = 0 \ \mbox{ on } \ \ndom \} $ (see also Subsection \ref{ssec:fs}).
In this paper, we are interested in the solutions to \eqref{eq} (or equivalently \eqref{vi}) and \eqref{bcic} defined as follows:

\begin{defi}\label{defi:sol}
	A function $ z \in C([0,T];L^2(\Omega)) $ is called a \emph{strong solution} on $ [0,T] $ to \eqref{eq} (or equivalently \eqref{vi}) and \eqref{bcic} if the following (i)--(iv) are all satisfied:
	\begin{enumerate}\itemsep=1mm
		\item $ z \in W^{1,2} (0,T;L^2(\Omega)) \cap L^2 (0,T;X \cap V) $;
		\item $ \partial_t z \leq 0 \, $, $ \, -\Delta z + \lambda z + \sigma \gamma(z) \leq f \, $ a.e.~in $ \, \Omega \times (0,T) $;
		\item $ \partial_t z \left( -\Delta z + \lambda z + \sigma \gamma(z) - f \right) = 0 \, $ a.e.~in $ \, \Omega \times (0,T) $;
		\item $ z(x,0) = z_0(x) \, $ for a.e.~$ x \in \Omega $.
	\end{enumerate}
	Moreover, a function $z \in C \left( [0,\infty);L^2(\Omega) \right)$ is called a \emph{global-in-time strong solution} to \eqref{eq} (or equivalently \eqref{vi}) and \eqref{bcic} if the following (v) and (vi) are satisfied:
	\begin{enumerate}\itemsep=1mm
	\setcounter{enumi}{4}
		\item $z \in W^{1,2}_\mathrm{loc} \left( [0,\infty);L^2(\Omega) \right) \cap L^2_\mathrm{loc} \left( [0,\infty);X \cap V \right)$;
		\item for all $ T>0 $, $ \left. z \right|_{[0,T]} $ is a strong solution to \eqref{eq} and \eqref{bcic} on $ [0,T] $.
	\end{enumerate}
\end{defi}

Let $ V^* $ and $ X^* $ denote the dual spaces of $ V $ and $ X $, respectively.
We introduce the further assumption concerning the regularity of the boundary $ \ddom $.
Define an operator $B:V \to V^*$ by
\[
	\left\langle Bu, v \right\rangle_V := \int_\Omega \left( \nabla u \cdot \nabla v + uv \right) \d x \quad \mbox{ for } \ u, v \in V.
\]
Throughout this paper, we assume that
\begin{equation}\label{ass}
	B^{-1} w \in H^2(\Omega) \cap V \quad \mbox{ for } \ w \in L^2(\Omega).
\end{equation}
The condition \eqref{ass} can be seen as an \textit{elliptic regularity} and it is well known that \eqref{ass} holds true for smooth domains along with either Dirichlet or Neumann boundary condition.
However, it is rather delicate for mixed boundary conditions as well as for non-smooth domains, such as Lipschitz domains (see \cite{lm, grisvard, ms}).
We take mixed boundary conditions into account in view of applications to fracture models (cf.~\cite{fm, giacomini}, see also Subsection \ref{ssec:reg}).

Main theorems of this paper will be stated below.
Our first theorem is concerned with the existence of solutions to \eqref{eq} (or equivalently, \eqref{vi}) and \eqref{bcic}.

\begin{thm}[Existence of solutions]\label{thm:main}
	Assume \eqref{ass} and let $ \lambda \geq 0 $, $ z_0 \in X \cap V $, $ f \in L^2(0,T;L^2(\Omega)) \cap W^{1,2}(0,T;V^*) $, $ \sigma \in L^\infty(\Omega \times (0,T)) \cap W^{1,2}(0,T;L^p(\Omega)) $ for $ p = \max \{ 1, N/2 \} $ if $ N \neq 2 $ and $ p>1 $ if $ N=2 $.
	Let $ \gamma: \R \to \R $ be a continuous function.
	Suppose that the following {\rm (i)}--{\rm (v)} are all fulfilled\/{\rm :}
	\begin{enumerate}
		\item There exists a constant $ L > 0 $ such that
			\begin{align}
				&(\gamma(t) - \gamma(s))(t-s) + L|t-s|^2 \geq 0 \quad \mbox{ for all } \ t,s \in \R,\label{ass:gammal}\\
				&\lambda_0 := \lambda - L \left\| \sigma \right\|_{L^\infty(\Omega \times (0,T))} > 0;\label{ass:sigma}
			\end{align}
		\item there exists a constant \( C_1 > 0 \) such that
			\begin{equation}\label{ass:gammau}
				\left| \gamma(s) \right| \leq C_1 \left( |s|+1 \right) \quad \mbox{ for all } \ s \in \R;
			\end{equation}
		\item it holds that
		\[
			-\Delta z_0 + \lambda z_0 + \sigma(\cdot,0) \gamma(z_0) \leq f(\cdot,0) \quad \mbox{ in } \ V^*;
		\]
		\item there exists a function $ \tilde{f} \in L^2(\Omega) $ such that
			\begin{equation}\label{ass:f}
				f(x,t) \geq \tilde{f}(x) \quad \mbox{ for a.e.~} (x,t) \in \Omega \times (0,T);
			\end{equation}
		\item the function $ \sigma $ satisfies $ \sigma \geq 0 $ a.e.~in $ \, \Omega \times (0,T) $.
	\end{enumerate}
	Then, there exists a strong solution $ z = z(x,t) $ to \eqref{eq} {\rm (}or equivalently \eqref{vi}{\rm )} and \eqref{bcic} such that $ z \in W^{1,2}(0,T;V) $.
\end{thm}

The next theorem shows that the solutions $ z=z(x,t) $ to \eqref{eq} and \eqref{bcic} fulfill three qualitative properties, which correspond to the aforementioned three conditions of quasistatic evolution.
We define an energy functional $ \E : V \times [0,T] \to \R $ by
\begin{equation}\label{eq:energy}
	\E(u,t) := \frac{1}{2} \left\| \nabla u \right\|^2_{L^2(\Omega)} + \frac{\lambda}{2} \left\| u \right\|^2_{L^2(\Omega)} + \int_\Omega \sigma(t) \hat{\gamma}(u) \, \d x - \left\langle f(t), u \right\rangle_V \quad \mbox{ for } \ u \in V,\ t \in [0,T],
\end{equation}
where $ \hat{\gamma} $ is the primitive function of $ \gamma $.

\begin{thm}[Qualitative properties of solutions]\label{thm:qua}
	Let $ f $, $ \sigma $ and $ z_0 $ satisfy all the assumptions in Theorem {\rm \ref{thm:main}} and let $ z \in W^{1,2}(0,T;L^2(\Omega)) \cap L^2(0,T;X \cap V) $ be a strong solution on $ [0,T] $ to \eqref{eq} {\rm (}or equivalently \eqref{vi}{\rm )} and \eqref{bcic}.
	Then $ z $ enjoys the following three conditions\/{\rm :}
	\begin{enumerate}
		\item {\rm (}{\bf Complete irreversibility of evolution}{\rm )} It holds that
		\[
			\partial_t z \leq 0 \quad \mbox{ a.e.~in } \ \Omega \times (0,T);
		\]
		\item {\rm (}{\bf Unilateral equilibrium of the energy}{\rm )} for all $t \in [0,T]$, it holds that
		\[
			\E (z(t),t) \leq \E (v,t)
		\]
		for any $v \in V$ satisfying $v \leq z(t)$ a.e.~in $\Omega${\rm ;}
		\item {\rm (}{\bf Energy balance law}{\rm )} the function $t \mapsto \E(z(t),t)$ is absolutely continuous on $[0,T]$ and the energy balance law holds in the following sense\/{\rm :}
		\[
			\E (z(t),t) - \E (z(s),s) = \int^t_s \int_\Omega \partial_t \sigma(r) \hat{\gamma}(z(r)) \, \d x \, \d r - \int^t_s \left\langle \partial_t f(r), z(r) \right\rangle_V \d r
		\]
		for all $t, s \in [0,T]$. In particular, if both $ f $ and $ \sigma $ are stationary in time, that is, $\partial_t f = \partial_t \sigma \equiv 0$, then no evolution occurs, i.e., $z(t) \equiv z_0$.
	\end{enumerate}
\end{thm}

The following two theorems are concerned with global-in-time strong solutions of \eqref{eq} and \eqref{bcic}.
The existence of global-in-time solutions to \eqref{eq} (or equivalently \eqref{vi}) and \eqref{bcic} is proved in Theorem \ref{thm:global}, and Theorem \ref{thm:asympt} is concerned with the long-time dynamics of global-in-time strong solutions to \eqref{eq} (or equivalently \eqref{vi}) and \eqref{bcic}.
More precisely, we verify that each global-in-time solution $ z(t) $ converges to a stationary limit $ z_\infty $ as $ t \to \infty $ under certain assumptions for $ f $ and $ \sigma $, and moreover, we shall characterize the limit as a solution to the stationary problem \eqref{limvi} (or \eqref{limeq}) and \eqref{limbc}.

\begin{thm}[Global existence]\label{thm:global}
	Let $f \in L^2_\loc ([0,\infty);L^2(\Omega))$, $z_0 \in X \cap V$ and $\sigma \in L^\infty(\Omega \times (0,\infty)) $.
	Assume that the following {\rm (i)--(vi)} are all satisfied\/{\rm :}
	\begin{enumerate}\itemsep=1mm
		\item There exists a constant $ L > 0 $ such that
			\begin{align*}
				&(\gamma(t) - \gamma(s))(t-s) + L|t-s|^2 \geq 0 \quad \mbox{ for all } \ t,s \in \R,\notag\\
				&L \left\| \sigma \right\|_{L^\infty(\Omega \times (0,\infty))} < \lambda;
			\end{align*}
		\item there exists a constant \( C_1>0 \) such that
			\[
				\left| \gamma(s) \right| \leq C_1 \left( |s|+1 \right) \ \mbox{ for all } \ s \in \R;
			\]
		\item it holds that $-\Delta z_0 + \lambda z_0 + \sigma(\cdot,0) \gamma(z_0) \leq f(\cdot,0) \,$ in $\, V^*${\rm ;}
		\item there exists $\tilde{f} \in L^2(\Omega)$ such that $\tilde{f}(x) \leq f(x,t) \,$ for a.e.~$(x,t) \in \Omega \times (0,\infty)${\rm ;}
		\item the function $ \sigma = \sigma (x,t) $ satisfies $\sigma(x,t) \geq 0 \,$ for a.e.~$ (x,t) \in  \Omega \times (0,\infty) ${\rm ;}
		\item the function $ f $ and $ \sigma $ satisfy $\partial_t f \in (L^1 \cap L^2)(0,\infty;V^*)$ and $\partial_t \sigma \in (L^1 \cap L^2)(0,\infty;L^p(\Omega))$ for $ p = \max \{ 1, N/2 \} $ if $ N \neq 2 $ and $ p>1 $ if $ N=2 $.
	\end{enumerate}
	Then there exists a global-in-time solution $z \in C \left( [0,\infty);L^2(\Omega) \right)$ to \eqref{eq} {\rm (}or equivalently \eqref{vi}{\rm )} and \eqref{bcic} satisfying
	\[
		z \in W^{1,2}_\loc \left( [0,\infty);V \right) \cap L^2_\loc \left( [0,\infty);X \right).
	\]
\end{thm}

\begin{thm}[Long-time dynamics of global-in-time strong solutions]\label{thm:asympt}
	In addition to the assumptions {\rm (i)}--{\rm (vi)} in Theorem {\rm \ref{thm:global}}, suppose that the following {\rm (vii)} and {\rm (viii)} hold\/{\rm :}
	\begin{enumerate}
		\setcounter{enumi}{6}
		\item $ \sigma  $ does not depend on the time variable $ t $, i.e., there exists a function $ \tilde{\sigma} \in L^\infty(\Omega) $ such that $ \sigma(x,t) = \tilde{\sigma}(x) $ for a.e.~$ (x,t) \in \Omega \times (0,T) ${\rm ;}
		\item there exists a function $ f_\infty \in L^2(\Omega) $ such that
			\[
				f(t) \to f_\infty \quad \mbox{ strongly in } \  L^2(\Omega) \quad \mbox{ as } \ t \to \infty.
			\]
	\end{enumerate}
	Let $ z = z(x,t) $ be the global-in-time strong solution to \eqref{eq} and \eqref{bcic} obtained by Theorem {\rm \ref{thm:global}} \emph{via the time-discretization scheme} {\rm (}see Section {\rm \ref{sec:main})}.
	Then there exists a function $z_\infty \in X \cap V$ such that
	\[
		z(t) \to z_\infty \quad \mbox{ strongly in } \ V \quad \mbox{ as } \ t \to \infty,
	\]
	and it satisfies that
	\[
		z_\infty \leq z_0, \quad -\Delta z_\infty + \lambda z_\infty + \tilde{\sigma} \gamma(z_\infty) \leq f_\infty \quad \mbox{ a.e.~in } \ \Omega.
	\]
	Moreover, assume that
	\begin{equation}\label{ass:fglobal}
		f_\infty \leq f(t) \quad \mbox{ a.e.~in } \ \Omega \quad \mbox{ for a.e.~} t \in (0,\infty).
	\end{equation}
	Then the limit function $ z_\infty \in X \cap V $ is characterized as a unique strong solution to
	\begin{equation}\label{limvi}
		\left\lbrace
		\begin{alignedat}{4}
		&z_\infty \leq z_0, \quad -\Delta z_\infty + \lambda z_\infty + \tilde{\sigma} \gamma(z_\infty) \leq f_\infty & \quad &\mbox{ in } \ \Omega,\\
		&(z_\infty - z_0) \left( -\Delta z_\infty + \lambda z_\infty + \tilde{\sigma} \gamma(z_\infty) - f_\infty \right) = 0 &&\mbox{ in } \ \Omega
		\end{alignedat}
		\right.
	\end{equation}
	equipped with the boundary condition
	\begin{equation}\label{limbc}
		\left\lbrace
		\begin{alignedat}{4}
			&z_\infty = 0 & \quad &\mbox{ on } \ \ddom,\\
			&\partial_\nu z_\infty = 0 &&\mbox{ on } \ \ndom.
		\end{alignedat}
		\right.
	\end{equation}
	In particular, the variational inequality \eqref{limvi} is equivalently rewritten as the following differential inclusion\/{\rm :}
	\begin{equation}\label{limeq}
		\dind \left( z_\infty - z_0 \right) - \Delta z_\infty + \lambda z_\infty + \tilde{\sigma} \gamma(z_\infty) \ni f_\infty \ \mbox{ in } \ \Omega.
	\end{equation}
\end{thm}

\noindent
{\bf Plan of the paper.}\,
The present paper consists of five sections.
In Section \ref{sec:pre}, we shall set up some notation and state some preliminary facts which will be used in later sections.
Section \ref{sec:main} is devoted to proving Theorem \ref{thm:main} via the so-called minimizing movement scheme.
In Section \ref{sec:qua}, we shall verify Theorem \ref{thm:qua}, where we give a proof which is independent of the construction of the solutions.
Section \ref{sec:global} concerns global dynamics of global-in-time solutions for the initial-boundary value problem.
More precisely, we shall first prove Theorem \ref{thm:global} and then verify Theorem \ref{thm:asympt} concerning convergence of each global-in-time solution to a stationary limit, which can be characterized as a solution to some variational inequality.

\section{Preliminaries}\label{sec:pre}
\subsection{Notation}\label{ssec:note}

For $ a, b \in \R $, let us denote the minimum of $ a $ and $ b $ by $ a \wedge b $, and the maximum by $ a \vee b $.
Let $ \vphi_+ := \vphi \vee 0 $ denote the positive part of $ \vphi $.
For $ d \in \N $, let $ \H^d \left( \cdot \right) $ denote the $ d $-dimensional Hausdorff measure.
We denote by $C$ a non-negative constant which is independent of the elements of the corresponding space or set and may vary from line to line.

For each Banach space $ E $, let $ E^* $ denote the dual space of $ E $.
Then we denote by $ \left\langle f,x \right\rangle_E $ the duality pairing of $ x \in E $ and $ f \in E^* $.
For an open interval $ I $ on $ \R $ and $1 \leq p \leq +\infty $, we denote by $ L^p \left( I;E \right) $ and $ W^{1,p} \left( I;E \right) $ the Bochner space and the Sobolev-Bochner space, respectively.
For $ f \in L^p(I;L^r(\Omega)) $, we often write $ f(t) := f(\cdot,t) \in L^r(\Omega) $.
Moreover, we simply denote $ \partial / \partial t $ by $ \partial_t $.
Furthermore, unless otherwise stated, we denote by $ q $ the Sobolev critical exponent, i.e., $ q := 2N/(N-2) $ if $ N \geq 3 $, while $ 1 \leq q < \infty $ if $ N=2 $, and $ q=\infty $ if $ N=1 $.
Let $ q' \in (1,\infty] $ be the H\"{o}lder conjugate exponent of $ q $.

We define the indicator function $I_K : \R \to \R \cup \left\{ +\infty \right\}$ supported over a closed convex set $ K \subset \R $ as
\begin{equation}\label{eq:indicator}
	I_K (x) := \left\{ \,
	\begin{alignedat}{4}
		&0 &\quad &\mbox{ if } \ x \in K,\\
		&{+\infty} &&\mbox{ otherwise.}
	\end{alignedat}
	\right.
\end{equation}
For $1 \leq p \leq +\infty$, we define an extension of the indicator function onto $L^p(\Omega)$ by
\[
	I_K (f) := \left\{ \,
	\begin{alignedat}{4}
		&0 &\quad &\mbox{ if } \ f(x) \in K \quad \mbox{ for a.e.} \ x \in \Omega,\\
		&{+\infty} &&\mbox{ otherwise}
	\end{alignedat}
	\right.
\]
for $f \in L^p(\Omega)$.

\subsection{Sobolev spaces for mixed boundary conditions}\label{ssec:fs}

Let $ N \in \N $ and $ \Omega \subset \R^N $ be a bounded domain with smooth boundary $\dom$.
Let $ \ddom $ and $ \ndom $ be relatively open subsets of $ \dom $ such that $ \ddom \cap \ndom = \emptyset $ and $ \H^{N-1}(\dom \setminus (\ddom \cup \ndom))=0 $.
We set
\begin{align*}
	V &:= \{ \, u \in H^1(\Omega) \, \colon \, \Tr \, u=0 \; \; \H^{N-1} \mbox{-a.e.~on } \ddom \, \},\\
	X &:= \{ \, u \in H^2(\Omega) \, \colon \, \partial_\nu u=0 \; \; \H^{N-1} \mbox{-a.e.~on } \ndom \, \},
\end{align*}
where $ \Tr $ stands for the trace operator on $ \dom $, $\nu$ denotes the outer unit normal vector field on $\dom$ and $ \partial_\nu = \nu \cdot \Tr \, (\nabla u) $ denotes the outer normal derivative on $\dom$.
We set $ \| v \|_V := \| v \|_{H^1(\Omega)} $ and $ \| w \|_X := \| w \|_{H^2(\Omega)} $ for $ v \in V $ and $ w \in X $, respectively.
Moreover, we write $f \leq g$ in $V^*$ for $ f,g \in V^*$ if $\left\langle f, v \right\rangle_V \leq \left\langle g, v \right\rangle_V$ holds for any $v \in V$ satisfying $v \geq 0$ a.e.~in $\Omega$.

\subsection{Assumptions for the domain and its boundary}\label{ssec:reg}

We define an operator $B : V \to V^*$ as
\[
	\left\langle Bu,v \right\rangle_V := \int_\Omega \left( \nabla u \cdot \nabla v + uv \right) \d x \quad \mbox{ for } \, u,v \in V.
\]
Then $B$ is a bijection from $V$ to $V^*$, and hence $B^{-1} : V^* \to V$ is well defined.
Throughout this paper, we assume that
\begin{equation}\label{eq:ass}
	B^{-1} w \in H^2(\Omega) \cap V \quad \mbox{ for } \ w \in L^2(\Omega),
\end{equation}
which is the $L^2$-elliptic regularity condition to guarantee the $H^2$-regularity for the domain of the Laplacian equipped with the boundary condition \eqref{bcic}.

\begin{rem}\label{rem:ass}
	Under the assumption \eqref{eq:ass}, one can deduce that
	\[
		B^{-1} w \in X \cap V \quad \mbox{ for } \ w \in L^2(\Omega);
	\]
	indeed, let $ w \in L^2(\Omega) $ and $ u = B^{-1} w \in H^2(\Omega) \cap V $.
	Then thanks to Green's formula, for $ \psi \in V $, we have
	\begin{align*}
		\left\langle B u, \psi \right\rangle_{V} &= \int_\Omega \left( \nabla u \cdot \nabla \psi + u \psi \right) \d x = \int_{\ndom} \left( \partial_\nu u \right) \psi \, \d \H^{N-1} + \int_\Omega \left( -\Delta u + u \right) \psi \, \d x = \int_\Omega w \psi \, \d x.
	\end{align*}
	Hence it follows that $ -\Delta u + u = w $ a.e.~in $ \Omega $, and moreover, we obtain
	\begin{align*}
		\int_{\ndom} \left( \partial_\nu u \right) \psi \, \d \H^{N-1} = 0,
	\end{align*}
	which implies that $ \partial_\nu u=0 $ $ \H^{N-1} $-a.e.~on $ \ndom $.
	Therefore $ u \in X $.
\end{rem}

\subsection{Subdifferential operator}\label{ssec:subdif}

Let $H$ be a real Hilbert space equipped with an inner product $(\cdot,\cdot)_H$ and let $\vphi : H \to (-\infty, +\infty]$ be a proper, i.e., $\varphi \not\equiv +\infty$, lower-semicontinuous convex functional, whose \emph{effective domain} is defined by $D(\vphi) = \{ u \in H \colon \vphi(u) < +\infty \}$. The \emph{subdifferential operator} $\partial \vphi : H \to 2^H$ of $\varphi$ is defined by
$$
	\partial \vphi(u) := \left\{ \, \xi \in H \colon \vphi(v)-\vphi(u) \geq ( \xi, v-u )_H \quad \mbox{for } \ v \in D(\vphi) \, \right\} \ \mbox{ for } \ u \in D(\vphi)
$$
and $\partial \varphi(u) := \emptyset$ for $u \not\in D(\varphi)$ with domain $D(\partial \vphi) := \{ u \in H \colon \partial \vphi(u) \neq \emptyset \}$.
It is well known that $ \partial \vphi $ is maximal monotone in $H \times H$ (see, e.g.,~\cite{brezisF, rockafellar} for more details).

One can easily show that the subdifferential operator of the indicator function $ I_K $ (see \eqref{eq:indicator}) for an interval $ K := (-\infty,a] \subset \R $ and $ a \in \R $ is represented by
\[
	\partial I_{(-\infty, a]}(s) = \left\lbrace
	\begin{alignedat}{4}
		&\{ 0 \} &\quad &\mbox{ if } \ s \in (-\infty,a),\\
		&[0,+\infty) &&\mbox{ if } \ s = a,\\
		&\emptyset &&\mbox{ otherwise}
	\end{alignedat}
	\right. \quad \mbox{ for } \ s \in \R.
\]
In particular, we have $ t \in \partial I_{(-\infty, a]}(s) $ if and only if $ t \geq 0 $, $ s \leq a $ and $ t(s-a) = 0 $ hold.

\subsection{A chain-rule formula}\label{ssec:chain}

Define $ \vphi : L^2(\Omega) \to [0,+\infty] $ by
\begin{equation}\label{eq:phi}
	\vphi(v) := \left\{
	\begin{alignedat}{4}
		&\frac{1}{2} \int_\Omega \left| \nabla v(x) \right|^2 \d x &\quad &\mbox{ if } \ v \in V,\\
		&{+\infty} &&\mbox{ if } \ v \in L^2(\Omega) \setminus V
	\end{alignedat}
	\right.
	\quad \mbox{ for } \ v \in L^2(\Omega),
\end{equation}
where $ V $ is a Hilbert space defined in Subsection \ref{ssec:fs}.
It is known that $ \vphi $ is convex and lower-semicontinuous in $L^2(\Omega)$.
Moreover, if the elliptic regularity condition \eqref{eq:ass} is satisfied, then
\begin{equation*}
	D(\partial \vphi) = X \cap V \quad \mbox{ and } \quad \partial \vphi(v) = -\Delta v \quad \mbox{ for } \ v \in X \cap V,
\end{equation*}
where $ \partial \vphi $ and $ D(\partial \vphi) $ denote the subdifferential of $ \vphi $ and the domain of $ \partial \vphi $, respectively (see Subsection \ref{ssec:subdif}).
Furthermore, let $u \in W^{1,2}(0,T;L^2(\Omega)) \cap L^2(0,T;X \cap V)$.
Then a function
\[
	\vphi (u(t)) = \frac{1}{2} \int_\Omega \left| \nabla u(x,t) \right|^2 \d x, \quad \ t \in [0,T]
\]
is absolutely continuous on $[0,T]$ and complies with the chain-rule formula, i.e., it holds that
\begin{equation}\label{eq:chain}
	\frac{\d}{\d t} \, \vphi (u(t)) = \left( \partial_t u(t), - \Delta u(t) \right)_{L^2(\Omega)} \quad \mbox{ for a.e.~} t \in (0,T),
\end{equation}
where $ (\cdot, \cdot)_{L^2(\Omega)} $ denotes the inner product of $ L^2(\Omega) $ (see also \cite{barbu1976,brezisF,brezis1971,lm}, \cite{ak}*{Lemmas 3.1 and 3.4} and \cite{e2ciec}*{Lemma 2.1}).

\subsection{Nemytskii operator and Fr\'{e}chet derivative}\label{ssec:frechet}

Let $ \gamma : \R \to \R $ be a continuous function satisfying \eqref{ass:gammau} and let $ \hat{\gamma} $ denote the primitive function of $ \gamma $.
Then for $ u \in L^r(\Omega) $ with $ 1 \leq r < \infty $, the map $ L^r(\Omega) \to L^r(\Omega) $, $ u \mapsto \gamma(u) $ is well defined and continuous.
Let $ \rho \in L^p(\Omega) $, where $ p = N/2 $ if $ N \geq 3 $, $ p $ is an arbitrary exponent on $ (1,+\infty) $ if $ N=2 $ and $ p=1 $ if $ N=1 $.
Set
\[
	\Gamma (u) := \int_\Omega \rho \hat{\gamma}(u) \, \d x \quad \mbox{ for } \ u \in V.
\]
Then $ |\Gamma(u)| < +\infty $ by virtue of the embedding $ V \hookrightarrow L^q(\Omega) $ for the Sobolev critical exponent $ q $.
Furthermore, the functional $ u \mapsto \Gamma(u) $ is G\^{a}teaux differentiable in $ V $, and the derivative $ D_G \Gamma : V \to V^* $ is given by
\begin{align*}
	\left\langle D_\mathrm{G}\Gamma(u), v \right\rangle_V = \left\langle \rho \gamma(u), v \right\rangle_V = \int_\Omega \rho \gamma(u) v \, \d x \quad \mbox{ for } \ u,v \in V.
\end{align*}
Since $ D_\mathrm{G} \Gamma $ is continuous in $ V $, $ \Gamma $ turns out to be Fr\'{e}chet differentiable and it follows that $ \Gamma' = D_\mathrm{G} \Gamma $, where $ \Gamma' $ denotes the Fr\'{e}chet derivative of $ \Gamma $.
In particular, $ \Gamma $ is of class $ C^1 $ in $ V $.

\section{The existence of solutions}\label{sec:main}

This section is devoted to prove Theorem \ref{thm:main}.
To this end, we construct strong solutions of initial-boundary value problem \eqref{eq} and \eqref{bcic}.
Our proof is mainly based on the so-called minimizing movement scheme, i.e., it will be divided into three steps: first, we introduce time-discretization of \eqref{eq} and construct $ L^2 $-solutions of the discretized problems inductively (Theorem \ref{thm:dis}).
The second step is to derive some a priori estimates of interpolants of these solutions, which ensure that the interpolants are bounded in each functional space (Lemmas \ref{lem:zk}--\ref{lem:zkl} and \ref{lem:strong}).
Then, we obtain continuous limits of them by virtue of weak and weak-star compactness, and we confirm that the limit complies with (i)--(iii) of Definition \ref{defi:sol} (Lemma \ref{lem:identity}).

Before proceeding to a proof, we give some comments regarding Theorem \ref{thm:main}.

\begin{rem}\label{rem:ex}
	\begin{enumerate}
		\item Roughly speaking, \eqref{ass:gammal} implies that the slope of $ \gamma $ is bounded from below.
		In particular, if $ \gamma $ is Lipschitz continuous, \eqref{ass:gammal} and \eqref{ass:gammau} are fulfilled.
		\item The assumption (iii) of Theorem \ref{thm:main} can be regarded as an admissible condition for the initial data, and it is necessary for existence of strong solutions to \eqref{eq} and \eqref{bcic}.
		Indeed, if there exists a strong solution $z = z(x,t)$ to \eqref{eq} and \eqref{bcic}, it follows from (ii) of Definition \ref{defi:sol} that
		\[
			\Phi(t) := \int_\Omega \nabla z(t) \cdot \nabla v \, \d x + \int_\Omega \sigma(t) \gamma(z(t)) v \, \d x - \langle f(t), v\rangle_V \leq 0, \quad \mbox{ a.e.} \ t \in (0,T)
		\]
		for every $v \in V$ satisfying $ v \geq 0 $ a.e.~in $ \Omega $.
		On the other hand, $\Phi$ is continuous on $[0,T]$ since $z \in C([0,T];V)$ (see \cite{lm}) and $f \in W^{1,2}(0,T;V^*) \subset C([0,T];V^*)$.
		Hence $ \Phi(t) \leq 0 $ for all $t \in [0,T]$; in particular, it follows that $ \Phi(0) \leq 0 $.
		\item If $f$ belongs to $W^{1,1}(0,T;L^2(\Omega))$, (iv) of Theorem \ref{thm:main} holds true with $\tilde{f} \in L^2(\Omega)$ given by $ \tilde{f}(x) := f(x,0) - \int^T_0 \left| \partial_t f (x,s) \right| \d s $ for $ x \in \Omega $.
		Indeed, by virtue of the fundamental theorem of calculus, we see that
		\[
			f(x,t) = f(x,0) + \int^t_0 \partial_t f(x,s) \,\d s \geq \tilde{f}(x) \quad \mbox{ for a.e.~} (x,t) \in \Omega \times (0,T).
		\]
		On the other hand, the assumptions of $ f $ in Theorem \ref{thm:main} is strictly weaker than that $ f \in W^{1,1}(0,T;\Omega) $ (see \cite{e2ciec}*{Remark 3.1}).
	\end{enumerate}
\end{rem}

\subsection{Time-discretized problem}\label{ssec:dis}

Let $ T>0 $ and let $ \Omega \subset \R^N $ be a smooth bounded domain complying with \eqref{eq:ass}.
As in Theorem \ref{thm:main}, let $ \lambda > 0 $, $ z_0 \in X \cap V $, $ f \in L^2(0,T;L^2(\Omega)) \cap W^{1,2}(0,T;V^*) $ and $ \sigma \in L^\infty(\Omega \times (0,T)) \cap W^{1,2}(0,T;L^p(\Omega)) $.
Moreover, let $ \gamma: \R \to \R $ be a continuous function.
In what follows, suppose that {\rm (i)}--{\rm (v)} in Theorem \ref{thm:main} are all fulfilled.

We introduce time-discretized problems of \eqref{eq} and \eqref{bcic}.
For $ m \in \N $, set $ \tau := T/m $ and define $ \{ f_k \} $ and $ \{ \sigma_k \} $ by
\begin{alignat}{4}
	f_k &:= \frac{1}{\tau} \int^{t_k}_{t_{k-1}} f(\cdot,s) \, \d s \, \in L^2(\Omega) & \quad & \mbox{ for } \ k=1,2,\dots,m,\label{def:fk2}\\
	f_0 &:= f(\cdot,0) \, \in V^*,\label{def:f02}\\
	\sigma_k &:= \frac{1}{\tau} \int^{t_k}_{t_{k-1}} \sigma(\cdot,s) \, \d s \, \in L^\infty(\Omega) && \mbox{ for } \ k=1,2,\dots,m, \label{def:sigmak}\\
	\sigma_0 &:= \sigma(\cdot,0) \, \in L^p(\Omega). \label{def:sigma0}
\end{alignat}
For $ k=1,\dots,m $, let us consider the following elliptic obstacle problem:
\begin{equation}\label{vi:dis}
	\left\lbrace
	\begin{alignedat}{4}
		&z_k \leq z_{k-1}, \quad -\Delta z_k + \lambda z_k + \sigma_k \gamma(z_k) \leq f_k & \quad &\mbox{ in } \ \Omega,\\
		&(z_k - z_{k-1}) \left( -\Delta z_k + \lambda z_k + \sigma_k \gamma(z_k) - f_k \right) = 0 &&\mbox{ in } \ \Omega,
	\end{alignedat}
	\right.
\end{equation}
which is equivalent to the differential inclusion
\begin{equation}\label{eq:dis}
	\dind \left( \frac{z_k - z_{k-1}}{\tau} \right) - \Delta z_k + \lambda z_k + \sigma_k \gamma(z_k) \ni f_k \quad \mbox{ in } \, \Omega,
\end{equation}
equipped with the mixed boundary condition
\begin{equation}\label{bcic:dis}
	\left\lbrace
	\begin{alignedat}{4}
		&z_k = 0 &\quad &\mbox{ on } \, \ddom,\\
		&\partial_\nu z_k = 0 &&\mbox{ on } \, \ndom.
	\end{alignedat}
	\right.
\end{equation}
We define strong solutions to \eqref{eq:dis} and \eqref{bcic:dis} as follows:

\begin{defi}\label{defi:dis}
	For each $ k = 1, \dots, m $ and given $ z_{k-1} $, a function $ u \in L^2(\Omega) $ is called a \emph{strong solution} to \eqref{eq:dis} (or equivalently \eqref{vi:dis}) equipped with the boundary condition \eqref{bcic:dis} if the following (i)--(iii) are all satisfied:
	\begin{enumerate}\itemsep=1mm
		\item $ u \in X \cap V $;
		\item $ u \leq z_{k-1} \ $ and $ \ -\Delta u + \lambda u + \sigma_k \gamma(u) \leq f_k $\quad a.e.~in $ \, \Omega $;
		\item $ \left( u - z_{k-1} \right) \left( -\Delta u + \lambda u + \sigma_k \gamma(u) - f_k \right) = 0 $\quad a.e.~in $ \, \Omega$.
	\end{enumerate}
\end{defi}

\begin{thm}\label{thm:dis}
	Under the assumptions as in Theorem {\rm \ref{thm:main}}, there exists a unique strong solution $ z_k \in X \cap V $ to \eqref{eq:dis} and \eqref{bcic:dis} for each $ k=1,\dots,m $.
	Moreover, the so-called \emph{Lewy-Stampacchia estimate} holds true, i.e.,
	\begin{equation}\label{eq:ls}
		f_k \wedge \left( -\Delta z_{k-1} + \lambda z_{k-1} + \sigma_k \gamma(z_k) \right) \leq -\Delta z_k + \lambda z_k + \sigma_k \gamma (z_k) \leq f_k \quad \mbox{ a.e.~in } \ \Omega.
	\end{equation}
	Furthermore, the problem \eqref{eq:dis} and \eqref{bcic:dis} complies with comparison principle with respect to $ z_{k-1} $ and $ f $, i.e., let $ z_k^j $ be a solution to \eqref{eq:dis} and \eqref{bcic:dis} with $ z_{k-1} = z_{k-1}^j $ and $ f_k = f_k^j $ for $ j = 1,2 $.
	Then it holds that
	\[
		z_k^1 \leq z_k^2 \quad \mbox{ a.e.~in } \ \Omega,
	\]
	provided that
	\[
		z_{k-1}^1 \leq z_{k-1}^2 \quad \mbox{ and } \quad f_k^1 \leq f_k^2 \quad \mbox{ a.e.~in } \ \Omega.
	\]
\end{thm}

Before proving Theorem \ref{thm:dis}, we characterize a strong solution to \eqref{eq:dis} for each $ k $ as a minimizer of the following functional $ J_k $ over a convex subset of $ V $.
For $ k = 1, \dots, m $, define $ J_k : V \to \R$ by
\[
	J_k (u) := \frac{1}{2} \left\| \nabla u \right\|^2_{L^2(\Omega)} + \frac{\lambda}{2} \left\| u \right\|^2_{L^2(\Omega)} + \int_\Omega \sigma_k \hat{\gamma}(u) \, \d x - \int_\Omega f_k u \, \d x \quad \mbox{ for } \ u \in V,
\]
where we set $ \hat{\gamma}(s) := \int^s_0 \gamma(r) \, \d r $ for $ s \in \R $.
Set
\[
	K[v] := \{ \, u \in V \colon u \leq v \quad \mbox{a.e.~in } \, \Omega \, \}
\]
for $ v \in V $.
Moreover, let $ A : V \to V^* $ be a linear operator defined by
\[
	\left\langle Au, v \right\rangle_V := \int_\Omega \nabla u \cdot \nabla v \, \d x \quad \mbox{ for } \ u, v \in V.
\]
It is well known that the functional $ \mathcal{S}: V \to \R $ defined by
\[
	\mathcal{S}(u) = \frac{1}{2} \left\| \nabla u \right\|^2_{L^2(\Omega)} + \frac{\lambda}{2} \left\| u \right\|^2_{L^2(\Omega)}, \quad u \in V
\]
is of class $ C^1 $ in $ V $, and it satisfies that
\[
	\left\langle \mathcal{S}'(u), v \right\rangle_V = \left\langle Au + \lambda u, v \right\rangle_V = \int_\Omega \nabla u \cdot \nabla v \, \d x + \lambda \int_\Omega u v \, \d x \quad \mbox{ for } \ u,v \in V,
\]
where $ \mathcal{S}' $ denotes the Fr\'{e}chet derivative of $ \mathcal{S} $.

\begin{lem}[cf.~\cite{ak}]\label{lem:mini}
	Let $ u \in V $ and \( k = 1,\dots,m \).
	Then the following {\rm (i)--(v)} are equivalent to each other\/{\rm :}
	\begin{enumerate}
		\item The function $ u $ is a strong solution to \eqref{eq:dis} {\rm (}or equivalently \eqref{vi:dis}{\rm )} and \eqref{bcic:dis}{\rm ;}
		\item The function $ u $ satisfies $ u \in K[z_{k-1}] $ and minimizes $ J_k $ over $ K[z_{k-1}] $, i.e.,
			\[
				J_k(u) \leq J_k(v) \quad \mbox{ for all } \ v \in K[z_{k-1}];
			\]
		\item The function $ u $ satisfies $ u \in K[z_{k-1}] $ and
			\begin{equation}\label{eq:ue}
				\left\langle Au + \lambda u + \sigma_k \gamma(u) - f_k, \, v - u \right\rangle_V \geq 0 \quad \mbox{ for all } \ v \in K[z_{k-1}];
			\end{equation}
		\item The function $ u $ satisfies $ u \in K[z_{k-1}] $ and for any $ w \in V $,
			\begin{align*}
				&\left\langle Au + \lambda u + \sigma_k \gamma(u) - f_k, \, w \right\rangle_V \leq 0 \quad \mbox{ if } \ w \geq 0 \mbox{ a.e.~in } \ \Omega,\\
				&\left\langle Au + \lambda u + \sigma_k \gamma(u) - f_k, \, z_{k-1} - u \right\rangle_V = 0.
			\end{align*}
	\end{enumerate}
\end{lem}

\begin{rem}\label{rem:conv}
	In spite of the noncovexity of $ \gamma $, thanks to (i) of Theorem \ref{thm:main}, $ J_k $ turns out to be convex on $ V $ for $ k = 1, \dots, m $; indeed, it holds that
	\begin{align*}
		J_k(u) &= \frac{1}{2} \left\| \nabla u \right\|^2_{L^2(\Omega)} + \frac{\lambda}{2} \left\| u \right\|^2_{L^2(\Omega)} + \int_\Omega \sigma_k \hat{\gamma}(u) \, \d x - \int_\Omega f_k u \, \d x\\
		&= \frac{1}{2} \left\| \nabla u \right\|^2_{L^2(\Omega)} + \frac{1}{2} \int_\Omega \left( \lambda - L \sigma_k \right) \left| u \right|^2 \d x + \int_\Omega \sigma_k \left( \hat{\gamma}(u) + \frac{L}2 \left| u \right|^2 \right) \d x - \int_\Omega f_k u \, \d x
	\end{align*}
	for $ u \in V $.
	Here $ \lambda - L \sigma_k \geq \lambda_0 > 0 $ a.e.~in $ \Omega $ by virtue of \eqref{ass:sigma} and $ \R \ni s \mapsto \hat{\gamma}(s) + L|s|^2/2 $ is convex on $ \R $ by \eqref{ass:gammal}.
\end{rem}

\begin{proof}[Proof of Lemma {\rm \ref{lem:mini}}]
	Let $ u \in V $.
	First, we show that (ii) and (iii) are equivalent to each other.
	Assume (ii) and let $ v \in K[z_{k-1}] $ and $ \delta \in (0,1) $.
	Since $ K[z_{k-1}] $ is convex and $ J_k $ is of class $ C^1 $ (see Subsection \ref{ssec:chain}), we see that
	\begin{align*}
		0 \leq \frac{J_k((1-\delta)u+\delta v) - J_k(u)}{\delta} \to \left\langle D_\mathrm{G} J_k(u), v-u \right\rangle_V = \left\langle J_k'(u), v-u \right\rangle_V \quad \mbox{ as } \ \delta \to 0_+,
	\end{align*}
	where $ D_\mathrm{G} J_k $ and $ J_k' $ denote the G\^{a}teaux derivative and the Fr\'{e}chet derivative of $ J_k $, respectively.
	Therefore
	\begin{align*}
		\left\langle A u + \lambda u + \sigma_k \gamma(u) - f_k, v-u \right\rangle_V = \left\langle J_k'(u), v-u \right\rangle_V \geq 0,
	\end{align*}
	and thus (iii) follows.
	On the other hand, assume (iii) and fix $ v \in K[z_{k-1}] $.
	Since $ J_k $ is convex (see Remark \ref{rem:conv}), it follows from (iii) that
	\begin{align*}
		J_k(u) - J_k(v) \leq \left\langle J_k'(u), u-v \right\rangle_V = \left\langle A u + \lambda u + \sigma_k \gamma(u) - f_k, u-v \right\rangle_V \leq 0,
	\end{align*}
	which implies (ii).
	
	We next prove that (iii) is equivalent to (i).
	If (i) holds, i.e., $ u $ is a strong solution to \eqref{eq:dis}, then it follows that $ u \in X \cap K[z_{k-1}] $ and $ Au = -\Delta u \in L^2(\Omega) $.
	Fix $ v \in K[z_{k-1}] $.
	From the definition of strong solutions to \eqref{eq:dis}, it follows that
	\begin{alignat*}{4}
		\left( -\Delta u + \lambda u + \sigma_k \gamma(u) - f_k \right) \left( v - z_{k-1} \right) &\geq 0 & \quad &\mbox{ a.e.~in } \ \Omega,\\
		\left( -\Delta u + \lambda u + \sigma_k \gamma(u) - f_k \right) \left( u - z_{k-1} \right) &= 0 & \quad & \mbox{ a.e.~in } \ \Omega.
	\end{alignat*}
	Therefore,
	\begin{align*}
		\left\langle Au + \lambda u + \sigma_k \gamma(u) - f_k, v - u \right\rangle_V &= \int_\Omega \left( -\Delta u + \lambda u + \sigma_k \gamma(u) - f_k \right) \left( v - u \right) \d x\\
		&= \int_\Omega \left( -\Delta u + \lambda u + \sigma_k \gamma(u) - f_k \right) \left( v - z_{k-1} \right) \d x\\
		&\quad + \int_\Omega \left( -\Delta u + \lambda u + \sigma_k \gamma(u) - f_k \right) \left( z_{k-1} - u \right) \d x\\
		&= \int_\Omega \left( -\Delta u + \lambda u + \sigma_k \gamma(u) - f_k \right) \left( v - z_{k-1} \right) \d x \geq 0,
	\end{align*}
	and hence (iii) follows.
	Conversely, let us show that (iii) implies (i).
	Define an auxiliary functional $ \tilde{J}_k : V \to \R $ as
	\[
		\tilde{J}_k(v) := \frac{1}{2} \left\| \nabla v \right\|^2_{L^2(\Omega)} + \frac{\lambda}{2} \left\| v \right\|^2_{L^2(\Omega)} + \int_\Omega \sigma_k \gamma(u) v \, \d x - \int_\Omega f_k v \, \d x \quad \mbox{ for } \ v \in V,
	\]
	where $ u \in K[z_{k-1}] $ is the minimizer of $ J_k $ over $ K[z_{k-1}] $, and consider the minimizing problem of $ \tilde{J}_k $ over $ K[z_{k-1}] $.
	By the use of the direct method and the strictly convexity of $ \tilde{J}_k $, one can obtain a unique minimizer $ \tilde{u} \in K[z_{k-1}] $ of \( \tilde{J}_k \) over $ K[z_{k-1}] $, and moreover, an elliptic regularity theory implies that $ \tilde{u} \in X $ (see e.g., \cite{ak} for more details).
	However, (iii) and the fact that $ J_k'(u) = \tilde{J}_k'(u) $ in $ V^* $ imply that $ u $ is also a minimizer of $ \tilde{J}_k $ over $ K[z_{k-1}] $.
	Therefore from the uniqueness of the minimizer we obtain $ u = \tilde{u} \in X $.
	Thus we obtain $ Au = -\Delta u $ and, for any $ \vphi \in V $ with $ \vphi \geq 0 $, we deduce that
	\begin{align*}
		 \int_\Omega \left( -\Delta u + \lambda u + \sigma_k \gamma(u) - f_k \right) \vphi \, \d x =  \left\langle Au + \lambda u + \sigma_k \gamma(u) - f_k, \vphi \right\rangle_V \leq 0
	\end{align*}
	by taking $ v = u - \vphi \in K[z_{k-1}] $ in \eqref{eq:ue}.
	From the arbitrariness of $ \vphi $, we obtain
	\begin{equation}\label{eq:negative}
		-\Delta u + \lambda u + \sigma_k \gamma(u) - f_k \leq 0 \quad \mbox{ a.e.~in } \ \Omega.
	\end{equation}
	Furthermore, it follows that
	\begin{equation}\label{eq:zk-1}
		\int_\Omega \left( -\Delta u + \lambda u + \sigma_k \gamma(u) - f_k \right) \left( z_{k-1}-u \right) \d x = 0.
	\end{equation}
	Indeed, we find that the left-hand side of \eqref{eq:zk-1} is nonnegative by putting $ v = z_{k-1} $ in \eqref{eq:ue}, and also nonpositive from $ u \in K[z_{k-1}] $ and \eqref{eq:negative}.
	Therefore $ u $ is a strong solution to \eqref{eq:dis} and \eqref{bcic:dis}, and thus (i) follows.
	
	Finally let us prove that (iii) is equivalent to (iv).
	Assume (iii) and let $ w \in V $ satisfy $ w \geq 0 $ a.e.~in $ \Omega $.
	Putting $ v = u - w \in K[z_{k-1}] $ in \eqref{eq:ue}, we have
	\[
		\left\langle Au + \lambda u + \sigma_k \gamma(u) - f_k, \, w \right\rangle_V \leq 0.
	\]
	Thus it follows from \eqref{eq:ue} that
	\[
		\left\langle Au + \lambda u + \sigma_k \gamma(u) - f_k, \, z_{k-1} - u \right\rangle_V = 0,
	\]
	hence (iv) follows.
	On the other hand, assume (iv) and let $ v \in K[z_{k-1}] $.
	Then
	\begin{align*}
		\mel{\left\langle Au + \lambda u + \sigma_k \gamma(u) - f_k, v - u \right\rangle_V}\\
		&= \left\langle Au + \lambda u + \sigma_k \gamma(u) - f_k, v - z_{k-1} \right\rangle_V + \left\langle Au + \lambda u + \sigma_k \gamma(u) - f_k, z_{k-1} - u \right\rangle_V\\
		&= \left\langle Au + \lambda u + \sigma_k \gamma(u) - f_k, v - z_{k-1} \right\rangle_V \geq 0,
	\end{align*}
	and hence (iii) follows.
	This completes the proof.
\end{proof}

Now, we are in a position to prove Theorem \ref{thm:dis}.

\begin{proof}[Proof of Theorem {\rm \ref{thm:dis}}]
	Since $ J_k $ is convex, bounded from below and coercive in $ V $, there exists a sequence $ (u^j) \subset K[z_{k-1}]$ satisfying
	\begin{alignat*}{4}
		\disp \lim_{j \to \infty} J_k(u^j) &= \inf_{v \in K[z_{k-1}]} J_k(v) > - \infty \quad \mbox{ and } \quad \disp \| u^j \|_V &\leq C \quad \mbox{ for } \ j \in \N.
	\end{alignat*}
	By virtue of the compact embedding $ V \included L^2(\Omega) $, the reflexivity of $ V $ and the weak closedness of $ K[z_{k-1}] $ in $ V $, there exist a (not relabeled) subsequence of $ (u^j) $ and $ z_k \in K[z_{k-1}] $ such that
	\[
		u^j \wto z_k \quad \mbox{ weakly in } \ V \quad \mbox{ and } \quad u^j \to z_k \quad \mbox{ strongly in } \ L^2(\Omega).
	\]
	Furthermore, one can deduce that $ z _k $ is a minimizer of $ J_k $ over $ K[z_{k-1}] $ from the weak lower semicontinuity in $ V $ of $ J_k $; indeed,
	\begin{align*}
		J_k(z_k) \leq \liminf_{j \to \infty} J_k(u^j) = \inf_{v \in K[z_{k-1}]} J_k(v).
	\end{align*}
	Therefore, thank to Lemma \ref{lem:mini}, $ z_k $ is a strong solution to \eqref{eq:dis}.
	Moreover, the uniqueness follows from the strictly convexity of $ J_k $.
	Furthermore, since \( z_k \) is a solution to \eqref{eq:dis} and \eqref{bcic:dis}, and \( \sigma_k \gamma(z_k) \in L^2(\Omega) \), by virtue of the Lewy-Stampacchia estimate demonstrated in \cite{e2ciec}, we can deduce that
	\[
		\left( f_k - \sigma_k \gamma(z_k) \right) \wedge \left( -\Delta z_{k-1} + \lambda z_{k-1} \right) \leq -\Delta z_k + \lambda z_k \leq f_k - \sigma_k \gamma(z_k) \quad \mbox{ a.e.~in } \ \Omega
	\]
	for each \( k=1,\dots,m \).
	
	In what follows, we prove the comparison principle.
	Fix $ 1 \leq k \leq m $ and let $ z^j_k \in X \cap V $ be a solution to \eqref{eq:dis} and \eqref{bcic:dis} with $ z_{k-1} = z^j_{k-1} $ and $ f_k = f^j_k $ for $ j=1,2 $.
	Assume that $ z^1_{k-1} \leq z^2_{k-1} $ and $ f^1_k \leq f^2_k $ a.e.~in $ \Omega $ and set
	\[
		\eta^j_k := \Delta z^j_k - \lambda z^j_k - \sigma_k \gamma(z^j_k) + f^j_k \in \dind \, (z^j_k - z^j_{k-1})
	\]
	for $ j=1,2 $.
	Noting that
	\[
		\eta^1_k - \eta^2_k - \Delta (z^1_k - z^2_k) + \lambda (z^1_k - z^2_k) + \sigma_k (\gamma(z^1_k) - \gamma(z^2_k)) = f^1_k - f^2_k
	\]
	and testing it by $ (z^1_k - z^2_k)_+ $, we have
	\begin{align}
		\mel{\left( \eta^1_k - \eta^2_k, (z^1_k - z^2_k)_+ \right)_{L^2(\Omega)} + \left( -\Delta (z^1_k - z^2_k), (z^1_k - z^2_k)_+ \right)_{L^2(\Omega)}} \notag \\
		&\quad + \lambda \left( z^1_k - z^2_k, (z^1_k - z^2_k)_+ \right)_{L^2(\Omega)} + \left( \sigma_k (\gamma(z^1_k) - \gamma(z^2_k)), (z^1_k - z^2_k)_+ \right)_{L^2(\Omega)}\notag \\
		&= \left( f^1_k - f^2_k, (z^1_k - z^2_k)_+ \right)_{L^2(\Omega)},\label{eq:positive}
	\end{align}
	where $ (z^1_k - z^2_k)_+ $ denotes the positive part of $ z^1_k - z^2_k $.
	Set $ [z^1_k \geq z^2_k] := \{ \, x \in \Omega \colon z^1_k(x) \geq z^2_k(x) \, \} $.
	Since
	\begin{align*}
		\eta^j_k \geq 0, \quad z^j_k \leq z^j_{k-1} \quad \mbox{ and } \quad \eta^j_k \, (z^j_k - z^j_{k-1}) = 0 \quad \mbox{ a.e.~in } \ \Omega
	\end{align*}
	for $ j=1,2 $, we deduce that
	\begin{align*}
		( \eta^1_k - \eta^2_k, (z^1_k - z^2_k)_+ )_{L^2(\Omega)} &= \int_{[z^1_k \geq z^2_k]} ( \eta^1_k - \eta^2_k ) ( z^1_k - z^2_k ) \, \d x\\
		&= \int_{[z^1_k \geq z^2_k]} \eta^1_k \, ( z^1_k - z^1_{k-1} ) \, \d x - \int_{[z^1_k \geq z^2_k]} \eta^2_k \, ( z^1_k - z^1_{k-1} ) \, \d x\\
		&\quad - \int_{[z^1_k \geq z^2_k]} \eta^1_k \, ( z^2_k - z^1_{k-1} ) \, \d x + \int_{[z^1_k \geq z^2_k]} \eta^2_k \, ( z^2_k - z^1_{k-1} ) \, \d x\\
		&\geq - \int_{[z^1_k \geq z^2_k]} \eta^1_k \, ( z^2_k - z^1_{k-1} ) \, \d x + \int_{[z^1_k \geq z^2_k]} \eta^2_k \, ( z^2_k - z^1_{k-1} ) \, \d x\\
		&\geq - \int_{[z^1_k \geq z^2_k]} \eta^1_k \, ( z^1_k - z^1_{k-1} ) \, \d x + \int_{[z^1_k \geq z^2_k]} \eta^2_k \, ( z^2_k - z^2_{k-1} ) \, \d x = 0,
	\end{align*}
	and by assumption,
	\[
		\left( f^1_k - f^2_k, (z^1_k - z^2_k)_+ \right)_{L^2(\Omega)} \leq 0.
	\]
	Set $ \beta(s) = \gamma(s) + Ls $ for $ s \in \R $.
	Then $ \beta $ turns out to be monotone by virtue of \eqref{ass:gammal}, and thus we have
	\begin{align*}
		\mel{\left( \sigma_k (\gamma(z^1_k) - \gamma(z^2_k)), (z^1_k - z^2_k)_+ \right)_{L^2(\Omega)}}\\
		&= \int_{[z^1_k \geq z^2_k]} \sigma_k \, (\beta(z^1_k) - \beta(z^2_k)) (z^1_k - z^2_k) \, \d x -L \int_{[z^1_k \geq z^2_k]} \sigma_k \, |z^1_k - z^2_k|^2 \, \d x\\
		&\geq -L \int_{[z^1_k \geq z^2_k]} \sigma_k \, |z^1_k - z^2_k|^2 \, \d x.
	\end{align*}
	Therefore one can deduce from \eqref{eq:positive} that
	\[
		\int_{[z^1_k \geq z^2_k]} \left| \nabla \left( z^1_k - z^2_k \right) \right|^2 \d x + \int_{[z^1_k \geq z^2_k]} \left( \lambda - L \sigma_k \right) |z^1_k - z^2_k|^2 \, \d x \leq 0.
	\]
	Moreover, it follows from \eqref{ass:sigma} that
	\[
		\left\| (z^1_k - z^2_k) \, \chi_{[z^1_k \geq z^2_k]} \right\|^2_{L^2(\Omega)} = 0.
	\]
	Thus we obtain $ z^1_k = z^2_k $ a.e.~in $ [z^1_k \geq z^2_k] $, which implies that $ z^1_k \leq z^2_k $ a.e.~in $ \Omega $.
	This completes the proof.
\end{proof}

\subsection{A priori estimates}\label{ssec:est}

Owing to Theorem \ref{thm:dis}, the discretized equation \eqref{eq:dis} admits a unique strong solution $ z_k \in X \cap V $ for each $ k = 1, \dots, m $.
In this subsection, we establish some a priori estimates of $ \{ z_k \} $ and derive some compactness results.

\begin{lem}\label{lem:zk}
	It holds that
	\begin{equation}\label{est:zk}
		\left\| z_k \right\|_V \leq \phi \left( \left\| \partial_t f \right\|_{L^1(0,T;V^*)} + \left\| \partial_t \sigma \right\|_{L^1(0,T;L^p(\Omega))} \right) \quad \mbox{ for } \ k = 1, \dots, m,
	\end{equation}
	where $ \phi : \R \to \R $ is defined by
	\begin{equation}\label{eq:defphi}
		\phi(s) := C(\alpha + s) e^{Cs} \quad \mbox{ for } \ s \in \R,
	\end{equation}
	and $ \alpha > 0 $ is a constant given by
	\begin{equation}\label{eq:defalpha}
		\alpha := \left( \left\| \sigma_0 \right\|_{L^p(\Omega)} + 1 \right) \left\| z_0 \right\|^2_V + \sup_{t \in [0,T]} \left( \left\| \sigma(t) \right\|^2_{L^p(\Omega)} + \left\| f(t) \right\|^2_{V^*} \right) + 1.
	\end{equation}
	Moreover, the constant $ C \geq 1 $ in \eqref{eq:defphi} may depend on $ C_1 $, $ L $ and $ \lambda_0 $ {\rm (}see \eqref{ass:gammal}--\eqref{ass:gammau}{\rm )}, however, it is independent of $ k $, $ \tau $ and $ T $.
\end{lem}

\begin{proof}
	Set \( \eta_l := \Delta z_l - \lambda z_l - \sigma_l \gamma(z_l) + f_l \in \dind \left( z_l - z_{l-1} \right) \) for \( 1 \leq k \leq m \).
	Testing it by \( z_l - z_{l-1} \), we see that
	\begin{align}
		\mel{\int_\Omega \eta_l \left( z_l - z_{l-1} \right) \d x + \int_\Omega \nabla z_l \cdot \nabla \left( z_l - z_{l-1} \right) \d x + \lambda \int_\Omega z_l \left( z_l - z_{l-1} \right) \d x}\notag\\
		&+ \int_\Omega \sigma_l \gamma(z_l) (z_l - z_{l-1}) \, \d x = \int_\Omega f_l \left( z_l - z_{l-1} \right) \d x.\label{eq:zk}
	\end{align}
	Since \( \eta_l \in \dind(z_l - z_{l-1}) \), the first term of the left-hand side is \( 0 \).
	Set $ \gamma_0 := \gamma(0) $, \( \beta(s) := \gamma(s) - \gamma_0 + Ls \) and $ \hat{\beta}(s) = \int^s_0 \beta(r) \, \d r $ for \( s \in \R \).
	Then \( \beta \) is monotone and \( \hat\beta \) is convex in \( \R \), and moreover, $ \beta(0)=\hat{\beta}(0)=0 $ and $ \hat{\beta}(s) \geq 0 $ for any $ s \in \R $.
	Therefore it follows that
	\begin{align*}
		\mel{\int_\Omega \sigma_l \gamma (z_l) (z_l - z_{l-1}) \, \d x}\\
		&= \int_\Omega \sigma_l \beta(z_l) (z_l - z_{l-1}) \, \d x + \gamma_0 \int_\Omega \sigma_l \left( z_l - z_{l-1} \right) \d x - L \int_\Omega \sigma_l z_l \left( z_l - z_{l-1} \right) \d x\\
		&\geq \int_\Omega \sigma_l (\hat{\beta}(z_l) - \hat{\beta}(z_{l-1})) \, \d x + \gamma_0 \int_\Omega \sigma_l \left( z_l - z_{l-1} \right) \d x - L \int_\Omega \sigma_l z_l \left( z_l - z_{l-1} \right) \d x\\
		&= \int_\Omega \big( \sigma_l \hat{\beta}(z_l) - \sigma_{l-1} \hat{\beta}(z_{l-1}) \big) \, \d x - \int_\Omega \left( \sigma_l - \sigma_{l-1} \right) \hat{\beta}(z_{l-1}) \, \d x\\
		&\quad + \gamma_0 \left( \int_\Omega (\sigma_l z_l - \sigma_{l-1} z_{l-1}) \, \d x - \int_\Omega \left( \sigma_l - \sigma_{l-1} \right) z_{l-1} \, \d x \right) - L \int_\Omega \sigma_l z_l \left( z_l - z_{l-1} \right) \d x.
	\end{align*}
	Furthermore, since $ \lambda - L \sigma_l \geq 0 $ a.e.~in $ \Omega $ and by the convexity, we note that
	\begin{align*}
		\mel{\int_\Omega \left( \lambda - L \sigma_l \right) z_l \left( z_l - z_{l-1} \right) \d x}\\
		&\geq \frac{1}{2} \left( \int_\Omega \left( \lambda - L \sigma_l \right) \left| z_l \right|^2 \d x - \int_\Omega \left( \lambda - L \sigma_l \right) \left| z_{l-1} \right|^2 \d x \right)\\
		&= \frac{1}{2} \left( \int_\Omega \left( \lambda - L \sigma_l \right) \left| z_l \right|^2 \d x - \int_\Omega \left( \lambda - L \sigma_{l-1} \right) \left| z_{l-1} \right|^2 \d x + L \int_\Omega \left( \sigma_l - \sigma_{l-1} \right) \left| z_{l-1} \right|^2 \d x \right).
	\end{align*}
	Therefore one can deduce from \eqref{eq:zk} that
	\begin{align}
		\mel{\frac{1}{2} \left( \left\| \nabla z_l \right\|^2_{L^2(\Omega)} - \left\| \nabla z_{l-1} \right\|^2_{L^2(\Omega)} \right) + \frac{1}{2} \left( \int_\Omega \left( \lambda - L \sigma_l \right) \left| z_l \right|^2 \d x - \int_\Omega \left( \lambda - L \sigma_{l-1} \right) \left| z_{l-1} \right|^2 \d x \right)}\notag\\
		&\quad + \frac{L}{2} \int_\Omega \left( \sigma_l - \sigma_{l-1} \right) \left| z_{l-1} \right|^2 \d x + \int_\Omega \big( \sigma_l \hat{\beta}(z_l) - \sigma_{l-1} \hat{\beta}(z_{l-1}) \big) \, \d x \notag\\
		&\quad - \int_\Omega \left( \sigma_l - \sigma_{l-1} \right) \hat{\beta}(z_{l-1}) \, \d x + \gamma_0 \left( \int_\Omega (\sigma_l z_l - \sigma_{l-1} z_{l-1}) \, \d x - \int_\Omega \left( \sigma_l - \sigma_{l-1} \right) z_{l-1} \, \d x \right) \notag\\
		&\leq \int_\Omega (f_l z_l - f_{l-1} z_{l-1}) \, \d x - \int_\Omega (f_l - f_{l-1}) z_{l-1} \, \d x. \label{eq:beta}
	\end{align}
	Let $ 1 \leq k \leq m $.
	By summing up \eqref{eq:beta} from $ l=1 $ to $ k $, we derive that
	\begin{align*}
		\mel{\frac{1}{2} \left( \left\| \nabla z_k \right\|^2_{L^2(\Omega)} + \int_\Omega \left( \lambda - L \sigma_k \right) \left| z_k \right|^2 \d x \right) + \int_\Omega \sigma_k \hat{\beta}(z_k) \, \d x + \gamma_0 \int_\Omega \sigma_k z_k \, \d x}\\
		&\leq \frac{1}{2} \left( \left\| \nabla z_0 \right\|^2_{L^2(\Omega)} + \int_\Omega \left( \lambda - L \sigma_0 \right) \left| z_0 \right|^2 \d x \right) + \int_\Omega \sigma_0 \hat{\beta}(z_0) \, \d x + \gamma_0 \int_\Omega \sigma_0 z_0 \, \d x\\
		&\quad + \left\langle f_k, z_k \right\rangle_V - \left\langle f_0, z_0 \right\rangle_V - \tau \sum^k_{l=1} \left\{ \frac{L}{2} \int_\Omega \frac{\sigma_l - \sigma_{l-1}}{\tau} \left| z_{l-1} \right|^2 \d x - \int_\Omega \frac{\sigma_l - \sigma_{l-1}}{\tau} \, \hat{\beta}(z_{l-1}) \, \d x \right. \notag\\
		&\quad - \left. \gamma_0 \int_\Omega \frac{\sigma_l-\sigma_{l-1}}{\tau} \, z_{l-1} \, \d x + \left\langle \frac{f_l - f_{l-1}}{\tau}, z_{l-1} \right\rangle_V\, \right\}.
	\end{align*}
	Moreover, since $ \hat{\beta} \geq 0 $ and
	\[
		\hat{\beta}(s) \leq C (s^2+1) \quad \mbox{ for } \ s \in \R
	\]
	(see \eqref{ass:gammau}), one can deduce that
	\begin{align*}
		\mel{\frac{1}{2} \left( \left\| \nabla z_k \right\|^2_{L^2(\Omega)} + \int_\Omega \left( \lambda - L \sigma_k \right) \left| z_k \right|^2 \d x \right) + \gamma_0 \int_\Omega \sigma_k z_k \, \d x - \left\langle f_k, z_k \right\rangle_V}\\
		&\leq C \left\{ (\left\| \sigma_0 \right\|_{L^p(\Omega)} + 1) \left\| z_0 \right\|^2_V + \left\| \sigma_0 \right\|^2_{L^{q'}(\Omega)} + \left\| f_0 \right\|^2_{V^*} + 1 \right\}\\
		&\quad + C \tau \sum^k_{l=1} \left( \left\| \frac{f_l - f_{l-1}}{\tau} \right\|_{V^*} + \left\| \frac{\sigma_l-\sigma_{l-1}}{\tau} \right\|_{L^p(\Omega)} \right) \left( 1 + \left\| z_{l-1} \right\|^2_V \right).
	\end{align*}
	Furthermore, it follows from \eqref{ass:sigma} that
	\[
		\left\| z_k \right\|^2_V \leq C \left( \left\| \nabla z_k \right\|^2_{L^2(\Omega)} + \int_\Omega \left( \lambda - L \sigma_k \right) \left| z_k \right|^2 \d x \right).
	\]
	Therefore we obtain
	\begin{align*}
		\left\| z_k \right\|^2_V &\leq C \, \Big\{ (\left\| \sigma_0 \right\|_{L^p(\Omega)} + 1) \left\| z_0 \right\|^2_V + \left\| \sigma_0 \right\|^2_{L^{q'}(\Omega)} + \left\| f_0 \right\|^2_{V^*}\\
		&\quad + \left. \left\| \sigma_k \right\|^2_{L^{q'}(\Omega)} + \left\| f_k \right\|^2_{V^*} + 1 + \tau \sum^m_{l=1} \left( \left\| \frac{f_l - f_{l-1}}{\tau} \right\|_{V^*} + \left\| \frac{\sigma_l-\sigma_{l-1}}{\tau} \right\|_{L^p(\Omega)} \right) \right\}\\
		&\quad + C \tau \sum^k_{l=1} \left( \left\| \frac{f_l - f_{l-1}}{\tau} \right\|_{V^*} + \left\| \frac{\sigma_l-\sigma_{l-1}}{\tau} \right\|_{L^p(\Omega)} \right) \left\| z_{l-1} \right\|^2_V.
	\end{align*}
	Using a Gronwall inequality of discrete type (see \cite{clark, corduneanu} for example), one can deduce that
	\begin{align*}
		\left\| z_k \right\|^2_V
		&\leq C \left( \alpha + \tau \sum^m_{l=1} \left( \left\| \frac{f_l - f_{l-1}}{\tau} \right\|_{V^*} + \left\| \frac{\sigma_l-\sigma_{l-1}}{\tau} \right\|_{L^p(\Omega)} \right) \right)\\
		&\quad \cdot \exp \left( C \tau \sum^m_{l=1} \left( \left\| \frac{f_l - f_{l-1}}{\tau} \right\|_{V^*} + \left\| \frac{\sigma_l-\sigma_{l-1}}{\tau} \right\|_{L^p(\Omega)} \right) \right)\\
		&= \phi \left( \tau \sum^m_{l=1} \left( \left\| \frac{f_l - f_{l-1}}{\tau} \right\|_{V^*} + \left\| \frac{\sigma_l-\sigma_{l-1}}{\tau} \right\|_{L^p(\Omega)} \right) \right)\\
		&\leq \phi \left( \left\| \partial_t f \right\|_{L^1(0,T;V^*)} + \left\| \partial_t \sigma \right\|_{L^1(0,T;L^p(\Omega))} \right)
	\end{align*}
	(see \eqref{eq:defphi} and \eqref{eq:defalpha} for the definition of $ \phi $).
	Therefore \eqref{est:zk} has been obtained.
\end{proof}

Let $ \phi_{f,\sigma,T} := \phi \, (\left\| \partial_t f \right\|_{L^1(0,T;V^*)} + \left\| \partial_t \sigma \right\|_{L^1(0,T;L^p(\Omega))}) \geq 1 $.
Then we further prove that

\begin{lem}\label{lem:zkd}
	There exists a constant $ C>0 $ such that
	\begin{equation}\label{est:zkd}
		\left\| \frac{z_k - z_{k-1}}{\tau} \right\|^2_V \leq C \phi_{f,\sigma,T}^2 \left( \left\| \frac{\sigma_k - \sigma_{k-1}}{\tau} \right\|^2_{L^p(\Omega)} + \left\| \frac{f_k - f_{k-1}}{\tau} \right\|^2_{V^*} \right)
	\end{equation}
	for all $ k = 1,\dots,m $.
	Moreover, the constant $ C $ is independent of $ k $, $ \tau $ and $ T $.
\end{lem}

\begin{proof}
	Set $ \eta_k := \Delta z_k - \lambda z_k - \sigma_k \gamma(z_k) + f_k \in L^2(\Omega) $ for each $ k = 1, \dots, m $ and set $ \eta_0 := \Delta z_0 - \lambda z_0 - \sigma(0) \gamma(z_0) + f(0) \in V^* $.
	Then we see that
	\[
		\eta_k \in \dind \left( z_k-z_{k-1} \right) \quad \mbox{ for } \ k = 1, 2, \dots, m,
	\]
	and for each $ k \geq 2 $, it follows that
	\[
		\eta_k - \eta_{k-1} - \Delta z_k + \Delta z_{k-1} + \lambda \left( z_k - z_{k-1} \right) + \sigma_k \gamma (z_k) - \sigma_{k-1} \gamma (z_{k-1}) = f_k - f_{k-1}
	\]
	in $ L^2(\Omega) $.
	Testing it by $ z_k - z_{k-1} $, we have
	\begin{align}
		\mel{\left( \eta_k - \eta_{k-1}, z_k - z_{k-1} \right)_{L^2(\Omega)} + \left\| \nabla (z_k-z_{k-1}) \right\|^2_{L^2(\Omega)} + \lambda \left\| z_k - z_{k-1} \right\|^2_{L^2(\Omega)}}\notag\\
		&\quad + \left( \sigma_k ( \gamma (z_k) - \gamma (z_{k-1}) ), z_k - z_{k-1} \right)_{L^2(\Omega)} + \left( \gamma (z_{k-1}) ( \sigma_k - \sigma_{k-1} ), z_k - z_{k-1} \right)_{L^2(\Omega)} \notag\\
		&= \left\langle f_k - f_{k-1}, z_k - z_{k-1} \right\rangle_V,\label{eq:test}
	\end{align}
	where $ \left( \cdot, \cdot \right)_{L^2(\Omega)} $ stands for the inner product of $ L^2(\Omega) $.
	The first term of the left-hand side of \eqref{eq:test} is nonnegative; indeed,
	\begin{align*}
		\left( \eta_k - \eta_{k-1}, z_k - z_{k-1} \right)_{L^2(\Omega)} &= \left( \eta_k , z_k - z_{k-1} \right)_{L^2(\Omega)} - \left( \eta_{k-1}, z_k - z_{k-1} \right)_{L^2(\Omega)}\\
		&= - \left( \eta_{k-1}, z_k - z_{k-1} \right)_{L^2(\Omega)} \geq 0 \quad \mbox{ for } \ k = 1, \dots, m.
	\end{align*}
	Moreover, set $ \gamma_0 := \gamma(0) $ and let $ \beta(s) := \gamma(s) - \gamma_0 + Ls $ for $ s \in \R $.
	Then $ \beta: \R \to \R $ is monotone, and hence, it follows that
	\begin{align*}
		\mel{\left( \sigma_k \left( \gamma (z_k) - \gamma (z_{k-1}) \right), \, z_k - z_{k-1} \right)_{L^2(\Omega)}}\\
		&= \left( \sigma_k \left( \beta (z_k) - \beta (z_{k-1}) \right), \, z_k - z_{k-1} \right)_{L^2(\Omega)} - L \int_\Omega \sigma_k \left| z_k - z_{k-1} \right|^2 \d x\\
		&\geq -L \int_\Omega \sigma_k \left| z_k - z_{k-1} \right|^2 \d x
	\end{align*}
	for $ k=1,\dots,m $.
	Furthermore, one can deduce from \eqref{ass:gammau} and \eqref{est:zk} that
	\begin{align*}
		\mel{\left| \left( \gamma (z_{k-1}) (\sigma_k - \sigma_{k-1}), z_k - z_{k-1} \right)_{L^2(\Omega)} \right|}\\
		&\leq \int_\Omega \left| \gamma (z_{k-1}) \right| \left| \sigma_k - \sigma_{k-1} \right| \left| z_k - z_{k-1} \right| \d x\\
		&\leq C \int_\Omega \left( \left| z_{k-1} \right| + 1 \right) \left| \sigma_k - \sigma_{k-1} \right| \left| z_k - z_{k-1} \right| \d x\\
		&\leq C \left( \left\| z_{k-1} \right\|_{L^{q}(\Omega)} + 1 \right) \left\| \sigma_k - \sigma_{k-1} \right\|_{L^p(\Omega)} \left\| z_k - z_{k-1} \right\|_{L^{q}(\Omega)}\\
		&\leq C \left( \left\| z_{k-1} \right\|_V + 1 \right) \left\| \sigma_k - \sigma_{k-1} \right\|_{L^p(\Omega)} \left\| z_k - z_{k-1} \right\|_V\\
		&\leq C \phi_{f,\sigma,T} \left\| \sigma_k - \sigma_{k-1} \right\|_{L^p(\Omega)} \left\| z_k - z_{k-1} \right\|_V,
	\end{align*}
	where $ C>0 $ is independent of $ k $, $ \tau $ and $ T $.
	Therefore \eqref{eq:test} implies that
	\begin{align*}
		\mel{\left\| \nabla \left( z_k - z_{k-1} \right) \right\|^2_{L^2(\Omega)} + \int_\Omega \left( \lambda - L \sigma_k \right) \left| z_k - z_{k-1} \right|^2 \d x}\\
		& \leq C \phi_{f,\sigma,T} \left\| \sigma_k - \sigma_{k-1} \right\|_{L^p(\Omega)} \left\| z_k - z_{k-1} \right\|_V + \left\| f_k - f_{k-1} \right\|_{V^*} \left\| z_k - z_{k-1} \right\|_V.
	\end{align*}
	Since $ \inf_{x \in \Omega} (\lambda - L \sigma_k) \geq \lambda_0 > 0 $ for any $ k $, we deduce by Young's inequality that
	\[
		\left\| z_k - z_{k-1} \right\|^2_V \leq C \phi_{f,\sigma,T}^2 \left( \left\| \sigma_k - \sigma_{k-1} \right\|^2_{L^p(\Omega)} + \left\| f_k - f_{k-1} \right\|^2_{V^*} \right),
	\]
	which gives \eqref{est:zkd}.
\end{proof}

\begin{lem}\label{lem:zkl}
	There exists a constant $ C>0 $ independent of $ \tau $ 
	such that
	\begin{equation}\label{est:zkl}
		\tau \sum^m_{k=1} \left\| z_k \right\|^2_X \leq C \left\{ \tau \sum^m_{k=1} \left( \left\| z_k \right\|^2_V + \left\| f_k \right\|^2_{L^2(\Omega)} \right) + \| \tilde{f} \|^2_{L^2(\Omega)} + \left\| z_0 \right\|^2_X + 1 \right\}.
	\end{equation}
\end{lem}

\begin{proof}
	One can deduce from \eqref{ass:f} and \eqref{eq:ls} that
	\begin{align*}
		f_k - \sigma_k \gamma(z_k) &\geq -\Delta z_k + \lambda z_k\\
		&\geq \left( f_k - \sigma_k \gamma(z_k) \right) \wedge \left( -\Delta z_{k-1} + \lambda z_{k-1} \right)\\
		&\geq \left( f_k - \sigma_k \gamma(z_k) \right) \wedge \left( f_{k-1} - \sigma_{k-1} \gamma(z_{k-1}) \right) \wedge \left( -\Delta z_{k-2} + \lambda z_{k-2} \right)\\
		&\geq \dots \geq \left( f_k - \sigma_k \gamma(z_k) \right) \wedge \dots \wedge \left( f_1 - \sigma_1 \gamma(z_1) \right) \wedge \left( -\Delta z_0 + \lambda z_0 \right)\\
		&\geq \big( \tilde{f} - \sigma_k \gamma(z_k) \big) \wedge \dots \wedge \big( \tilde{f} - \sigma_1 \gamma(z_1) \big) \wedge \left( -\Delta z_0 + \lambda z_0 \right) \quad \mbox{ a.e.~in } \ \Omega.
	\end{align*}
	for \( k=1,\dots,m \).
	Set $ \beta(s) = \gamma(s) + Ls $ for $ s \in \R $.
	Then by \eqref{ass:gammal}, $ \beta $ is monotone.
	Since $ z_k \leq z_{k-1} \leq \dots \leq z_0 $ holds a.e.~in $ \Omega $, we can deduce that
	\begin{align*}
		\tilde{f} - \sigma_l \gamma(z_l) &= \tilde{f} - \sigma_l \beta(z_l) + L \sigma_l z_l\\
		&\geq \tilde{f} - \sigma_l \beta(z_0) + L \sigma_l z_k\\
		&\geq \tilde{f} - \left\| \sigma \right\|_{L^\infty(I \times (0,T))} |\beta(z_0)| - L \left\| \sigma \right\|_{L^\infty(I \times (0,T))} |z_k| \quad \mbox{ a.e.~in } \ \Omega
	\end{align*}
	for $ 1 \leq l \leq k $.
	Thus we have
	\begin{align*}
		\left| -\Delta z_k + \lambda z_k \right| &\leq \left| f_k - \sigma_k \gamma(z_k) \right| + \left| \tilde{f} - \left\| \sigma \right\|_{L^\infty(I \times (0,T))} |\beta(z_0)| - L \left\| \sigma \right\|_{L^\infty(I \times (0,T))} |z_k| \right|\\
		&\quad + \left| -\Delta z_0 + \lambda z_0 \right| \quad \mbox{ a.e.~in } \ \Omega
	\end{align*}
	for $ k=1,\dots,m $.
	Therefore it follows from \eqref{ass:gammau} that
	\begin{align}
		\left\| \Delta z_k \right\|^2_{L^2(\Omega)} &\leq C \left( \left\| z_k \right\|^2_V + \left\| f_k \right\|^2_{L^2(\Omega)} + \| \tilde{f} \|^2_{L^2(\Omega)} + \left\| z_0 \right\|^2_X + 1 \right) \label{eq:zkl}
	\end{align}
	for each $ k $.
	Hence \eqref{est:zkl} has been proved by multiplying \eqref{eq:zkl} by $ \tau $ and summing up them from $ k = 1 $ to $ m $.
\end{proof}

Define \emph{piecewise linear} and \emph{constant interpolants} of $ \{ f_k \} $, $ \{ \sigma_k \} $ and $ \{ z_k \} $ by
\begin{alignat}{4}
	\ftau(x,t) &:= f_{k-1}(x) + \frac{t-t_{k-1}}{\tau} \left( f_k(x)-f_{k-1}(x) \right),\label{def:ftau}\\
	\ftaubar(x,t) &:= f_k(x),\label{def:ftaubar}\\
	\sigmatau (x,t) &:= \sigma_{k-1}(x) + \frac{t-t_{k-1}}{\tau} \left( \sigma_k(x)-\sigma_{k-1}(x) \right), \label{def:sigmatau}\\
	\sigmataubar (x,t) &:= \sigma_k(x), \label{def:sigmataubar}
\end{alignat}
and
\begin{align}
	\ztau(x,t) &:= z_{k-1}(x) + \frac{t-t_{k-1}}{\tau} \left( z_k(x)-z_{k-1}(x) \right),\label{def:ztau}\\
	\ztaubar(x,t) &:= z_k(x)\label{def:ztaubar}
\end{align}
for $ x \in \Omega $, $ t_{k-1} < t \leq t_k $ and $ k=1,\dots,m $.
Then it follows that
\begin{alignat*}{4}
	\ftau &\in L^\infty(0,T;L^2(\Omega)) \cap W^{1,2} (0,T;V^*), &&\ftaubar &&\in L^\infty (0,T;L^2(\Omega)),\\
	\sigmatau &\in L^\infty(\Omega \times (0,T)) \cap W^{1,2}(0,T;L^p(\Omega)), &&\sigmataubar &&\in L^\infty(\Omega \times (0,T)),\\
	\ztau &\in W^{1,2} (0,T;X \cap V), & \ \quad \ &\ztaubar &&\in L^\infty (0,T;X \cap V)
\end{alignat*}
for each $ \tau $.

\begin{rem}\label{rem:rephrased}
	Combined with \eqref{def:ftau}--\eqref{def:ztaubar}, the discretized equations \eqref{eq:dis} and the boundary conditions \eqref{bcic:dis} are rephrased as
	\begin{equation}\label{eq:reph}
		\left\lbrace
		\begin{alignedat}{4}
			&\dind \left( \partial_t \ztau \right) - \Delta \ztaubar + \lambda \ztaubar + \sigmataubar \gamma (\ztaubar) \ni \ftaubar & \quad &\mbox{ in } \ \Omega \times (0,T),\\
			&\ztau = \ztaubar = 0 &&\mbox{ on } \ \ddom \times (0,T),\\
			&\partial_\nu \ztau = \partial_\nu \ztaubar = 0 &&\mbox{ on } \ \ndom \times (0,T),\\
			&\ztau(\cdot,0) = z_0 &&\mbox{ in } \ \Omega.
		\end{alignedat}
		\right.
	\end{equation}
	Moreover, a priori estimates obtained in Lemmas \ref{lem:zk}--\ref{lem:zkl} can be rephrased as
	\begin{align*}
		\left\| \ztaubar \right\|_{L^\infty(0,T;V)} &\leq \phi_{f,\sigma,T},\\
		\left\| \partial_t \ztau \right\|^2_{L^2(0,T;V)} &\leq C \phi_{f,\sigma,T}^2 \left( \left\| \partial_t \ftau \right\|^2_{L^2(0,T;V^*)} + \left\| \partial_t \sigmatau \right\|^2_{L^2(0,T;L^p(\Omega))} \right),\\
		\left\| \ztaubar \right\|^2_{L^2(0,T;X)} &\leq C \left( \left\| \ztaubar \right\|^2_{L^2(0,T;V)} + \left\| \ftaubar \right\|^2_{L^2(0,T;L^2(\Omega))} + \| \tilde{f} \|^2_{L^2(\Omega)} + \left\| z_0 \right\|^2_X + 1 \right)\\
		&\leq C \left( T \phi^2_{f,\sigma,T} + \left\| f \right\|^2_{L^2(0,T;L^2(\Omega))} + \| \tilde{f} \|^2_{L^2(\Omega)} + \left\| z_0 \right\|^2_X + 1 \right).
	\end{align*}
	As for the last inequality, see Proposition \ref{pro:f} in Appendix \ref{appendix:inter}.
\end{rem}

\begin{lem}\label{lem:strong}
	The family $(\ztau)$ is bounded in $L^2(0,T;X)$, and moreover, it is precompact in $ C([0,T];V) $.
\end{lem}

\begin{proof}
	We find by convexity that
	\[
		\left\| \Delta z_\tau(t) \right\|_{L^2(\Omega)}^2 \leq \frac{t_k-t}{\tau} \left\| \Delta z_{k-1} \right\|_{L^2(\Omega)}^2 + \frac{t-t_{k-1}}{\tau} \left\| \Delta z_k \right\|_{L^2(\Omega)}^2 \quad \mbox{ for a.e.~} t \in (t_{k-1},t_k),
	\]
	which yields
	\[
		\int^{t_k}_{t_{k-1}} \|\Delta z_\tau(t)\|_{L^2(\Omega)}^2 \, \d t \leq \frac{\tau}2 \|\Delta z_{k-1}\|_{L^2(\Omega)}^2 + \frac{\tau}2 \|\Delta z_k\|_{L^2(\Omega)}^2.
	\]
	Hence it follows that
	\begin{align*}
		\int^T_0 \|\Delta z_\tau(t)\|_{L^2(\Omega)}^2 \, \d t &= \sum_{k=1}^m \int^{t_k}_{t_{k-1}} \|\Delta z_\tau(t)\|_{L^2(\Omega)}^2 \, \d t\\
		&\leq \frac{\tau}2 \sum_{k=1}^m \|\Delta z_{k-1}\|_{L^2(\Omega)}^2 + \frac{\tau}2 \sum_{k=1}^m\|\Delta z_k\|_{L^2(\Omega)}^2\\
		&\leq \frac{\tau}2 \|\Delta z_0\|_{L^2(\Omega)}^2 + \int^T_0 \|\Delta \ztaubar(t)\|_{L^2(\Omega)}^2 \, \d t \leq C.
	\end{align*}
	
	To show the precompactness of $ (\ztau) $ in $ C([0,T];V) $, we first prove the equi-continuity of $(\ztau)$ in $C([0,T];V)$.
	Using Lemma \ref{lem:zkd}, one observes that
	\begin{align*}
		\left\| \ztau(t) - \ztau(s) \right\|_V &\leq \int^t_s \left\| \partial_t \ztau(r) \right\|_V \d r\\
		&\leq \left| t-s \right|^{1/2} \left\| \partial_t \ztau \right\|_{L^2(0,T;V)} \leq C \left| t-s \right|^{1/2}
	\end{align*}
	for any $t ,s \in [0,T]$. Hence since $(\ztau)$ is bounded in $L^2(0,T;X)$ and $X$ is compactly embedded in $V$, using Theorem 3 of~\cite{simon}, one can conclude that $(\ztau)$ is relatively compact in $C([0,T];V)$. This completes the proof.
\end{proof}

\begin{lem}\label{lem:ztztb}
	It holds that
	\begin{equation}\label{eq:ztztb}
		\left\| \ztau - \ztaubar \right\|_{L^\infty(0,T;V)} \to 0 \quad \mbox{ as } \ \tau \to 0_+.
	\end{equation}
\end{lem}

\begin{proof}
	We observe from \eqref{est:zkd} and Proposition \ref{prop:inter} that
	\begin{align*}
		\left\| \ztau(t) - \ztaubar(t) \right\|_V &= \left\| z_{k-1} + \frac{t-t_{k-1}}{\tau} \left( z_k - z_{k-1} \right) - z_k \right\|_V\\
		&= \left| t-t_k \right| \left\| \frac{z_k-z_{k-1}}{\tau} \right\|_V\\
		&\leq C \tau \left( \left\| \frac{f_k-f_{k-1}}{\tau} \right\|_{V^*} + \left\| \frac{\sigma_k - \sigma_{k-1}}{\tau} \right\|_{L^p(\Omega)} \right)\\
		&\leq C \tau^{1/2} \left( \left\| \partial_t f \right\|_{L^2(0,T;V^*)} + \left\| \partial_t \sigma \right\|_{L^2(0,T;L^p(\Omega))} \right)
	\end{align*}
	for each $t \in (t_{k-1}, t_k]$ and $k = 1,\ldots,m$, and hence, the convergence follows.
\end{proof}

\subsection{Identification of the weak limits}\label{ssec:id}

Lemmas \ref{lem:zkd}, \ref{lem:zkl} and \ref{lem:strong} imply that there exist a (not relabeled) subsequence of $ (\tau) $ and a function $ z \in W^{1,2}(0,T;V) \cap L^2(0,T;X) $ such that
\begin{alignat}{4}
	z_\tau &\to z & \quad &\mbox{ weakly in } \ W^{1,2} (0,T;V),\label{eq:conv1}\\
	&&&\mbox{ strongly in } \ C ([0,T];V),\\
	\overline{z}_\tau &\to z &&\mbox{ weakly in } \ L^2 (0,T;X),\\
	&&&\mbox{ strongly in } \ L^\infty (0,T;V).\label{eq:conv4}
\end{alignat}
Here we also used \eqref{eq:ztztb}.
We prove that the limit function $ z=z(x,t) $ complies with (i)--(iv) of Definition \ref{defi:sol}, i.e., $ z $ is a strong solution to \eqref{eq} and \eqref{bcic}.

\begin{lem}\label{lem:identity}
	The weak limit $ z \in W^{1,2}(0,T;V) \cap L^2(0,T;X) $ of $ (\ztau) $ and $ (\ztaubar) $ exhibited in \eqref{eq:conv1}--\eqref{eq:conv4} is a strong solution of \eqref{eq} {\rm (}or equivalently \eqref{vi}{\rm )} and \eqref{bcic}.
\end{lem}

\begin{proof}
	We have already proved that $z \in W^{1,2} (0,T;V) \cap L^2 (0,T;X)$, which implies (i) of Definition \ref{defi:sol}.
	Moreover, we can deduce from Lemma \ref{lem:strong} that
	\begin{align*}
		\left\| z(0)-z_0 \right\|_{L^2(\Omega)} &= \left\| z(0)-z_\tau(0) \right\|_{L^2(\Omega)} \leq \sup_{t \in [0,T]} \left\| z(t) - z_\tau(t) \right\|_{L^2(\Omega)} \to 0 \quad \mbox{ as } \ \tau \to 0_+,
	\end{align*}
	which assures the condition (iv) of Definition \ref{defi:sol}.
	
	It still remains to check the conditions (ii) and (iii).
	Since $z_k$ is a strong solution to \eqref{eq:dis} for each $ k $, we deduce that
	\begin{alignat}{4}
		-\Delta \ztaubar + \lambda \ztaubar + \sigmataubar \gamma(\ztaubar) \leq \ftaubar, \quad \partial_t \ztau &\leq 0 & \quad &\mbox{ a.e.~in } \ \Omega \times (0,T), \notag \\
		\left( -\Delta \ztaubar + \lambda \ztaubar + \sigmataubar \gamma(\ztaubar) - \ftaubar \right) \partial_t \ztau &= 0 &&\mbox{ a.e.~in } \ \Omega \times (0,T) \label{eq:eqtau}
	\end{alignat}
	(see also \eqref{eq:reph}).
	Note that
	\begin{alignat}{4}
		\ftaubar &\to f & \quad & \mbox{ strongly in } \ L^2 (0,T;L^2(\Omega)) \label{ftau-conv},\\
		\sigmataubar &\to \sigma && \mbox{ strongly in } \ L^p(0,T;L^\infty(\Omega))\label{sigmatau-conv}
	\end{alignat}
	for any $ 1 \leq p < \infty $ as $ \tau \to 0_+ $ (see Proposition \ref{pro:f}).
	It follows from \eqref{eq:conv4} along with $ q'<q $ that $ \ztaubar(t) \to z(t) $ strongly in $ L^{q'}(\Omega) $ for a.e.~$ t \in (0,T) $, and thus we have
	\[
		\gamma(\ztaubar(t)) \to \gamma(z(t)) \quad \mbox{ strongly in } \ L^{q'}(\Omega) \quad \mbox{ for a.e.~} t \in (0,T)
	\]
	(see Subsection \ref{ssec:frechet}).
	Moreover, since \eqref{ass:gammau} implies that
	\begin{align*}
		\left\| \gamma (\ztaubar(t)) \right\|^{q'}_{L^{q'}(\Omega)} &\leq C_1 \, \Big( \left\| \ztaubar(t) \right\|^{q'}_{L^{q'}(\Omega)} + 1 \Big) \leq C_1 \, \Big( \left\| z \right\|^{q'}_{L^\infty(0,T;L^{q'}(\Omega))} + 1 \Big) < +\infty,
	\end{align*}
	one can deduce by the use of the dominated convergence theorem that
	\[
		\gamma(\ztaubar) \to \gamma(z) \quad \mbox{ strongly in } \ L^r(0,T;L^{q'}(\Omega))
	\]
	for any $ 1 \leq r < +\infty $ (for more details, see \cite{ch}).
	In particular, it follows from \eqref{sigmatau-conv} that
	\begin{equation}\label{eq:sigmagamma}
		\overline{\sigma}_\tau \gamma(\ztaubar) \to \sigma \gamma(z) \quad \mbox{ strongly in } \ L^2(0,T;L^{q'}(\Omega))
	\end{equation}
	as $ \tau \to 0_+ $.
	From the convergences obtained so far, it follows that
	\begin{equation}\label{eq:ineq}
		-\Delta z + \lambda z + \sigma \gamma(z) \leq f, \quad \partial_t z \leq 0 \quad \mbox{ a.e.~in } \ \Omega \times (0,T),
	\end{equation}
	which implies (ii). To prove (iii), we claim that
	\[
		0 =\int^T_0 \int_\Omega \left( \Delta z - \lambda z -\sigma \gamma(z) + f \right) \partial_t z \, \d x \, \d t.
	\]
	Indeed, recalling \eqref{def:ztau} and \eqref{def:ztaubar}, we observe that
	\begin{align*}
		\int^T_0 \int_\Omega \left( \Delta \overline{z}_{\tau} \right) \partial_t z_{\tau} \, \d x \, \d t 
		&= \sum^m_{k=1} \int_\Omega -\Delta z_k \left( -z_k + z_{k-1} \right) \d x\\
		&\leq \sum^m_{k=1} \int_\Omega \left( -\frac{1}{2} \left| \nabla z_k \right|^2 + \frac{1}{2} \left| \nabla z_{k-1} \right|^2 \right) \d x = -\vphi (z_{\tau}(T)) + \vphi (z_0),
	\end{align*}
	where $\vphi : L^2(\Omega) \to [0, +\infty]$ is the functional defined by \eqref{eq:phi}. Therefore, by the continuity of $\vphi$ in $V$ (see Subsection \ref{ssec:chain}), since $z_\tau(T) \to z(T)$ strongly in $V$, it follows that
	\begin{align*}
		\limsup_{\tau \to 0_+} \int^T_0 \int_\Omega \left( \Delta \overline{z}_\tau \right) \partial_t z_\tau \, \d x \,\d t 
		&\leq \limsup_{\tau \to 0_+} \left\{ -\vphi \left( z_\tau(T) \right) + \vphi (z_0) \right\}\\
		&= -\vphi (z(T)) + \vphi(z_0) = -\int^T_0 \frac{\d}{\d t} \, \vphi(z(t)) \,\d t\\
		&\stackrel{\text{\eqref{eq:chain}}}{=} \int^T_0 \int_\Omega \left(\Delta z\right) \partial_t z \, \d x \, \d t.
	\end{align*}
	Moreover, we find from \eqref{ftau-conv} that
	\[
		\lim_{\tau \to 0_+} \int^T_0 \int_\Omega \overline{f}_\tau \partial_t z_\tau \,\d x \,\d t = \int^T_0 \int_\Omega f \partial_t z \, \d x \, \d t.
	\]
	Furthermore, from \eqref{eq:sigmagamma}, we have
	\[
		\lim_{\tau \to 0_+} \int^T_0 \int_\Omega \overline{\sigma}_\tau \gamma(\ztaubar) \partial_t z_\tau \, \d x \, \d t = \int^T_0 \int_\Omega \sigma \gamma(z) \partial_t z \, \d x \, \d t.
	\]
	Combining all these facts, we obtain
	\begin{align*}
		0 &\overset{\text{\eqref{eq:eqtau}}}{=} \limsup_{\tau \to 0_+} \int^T_0 \int_\Omega \left( \Delta \overline{z}_\tau - \lambda \overline{z}_\tau - \sigma \gamma(\overline{z}_\tau) + \overline{f}_\tau \right) \partial_t z_\tau \, \d x \, \d t\\
		&\leq \int^T_0 \int_\Omega \left( \Delta z - \lambda z - \sigma \gamma(z) + f \right) \partial_t z \, \d x \, \d t \stackrel{\eqref{eq:ineq}}\leq 0.
	\end{align*}
	Hence (iii) follows, for the integrand is non-positive a.e.~in $\Omega \times (0,T)$ (see \eqref{eq:ineq}).
	Therefore the limit $ z $ complies with (i)--(iv) of Definition \ref{defi:sol}, where Theorem \ref{thm:main} has been proved.
\end{proof}

\begin{rem}[Comparison principle for time-dependent equation] \label{rem:comp}
	By virtue of Theorem \ref{thm:dis}, we can prove that \eqref{eq} and \eqref{bcic} comply with the comparison principle with respect to $ f $ and $ z_0 $, as long as the solution is restricted to be one obtained in Theorem \ref{thm:main}.
	In particular, the solution of \eqref{eq} and \eqref{bcic} which is obtained via \textit{the time-discretization scheme} is unique.
	On the other hand, the strategy depends on the way to construct solutions; thus we cannot directly apply the argument to show the uniqueness of solutions to \eqref{eq} and \eqref{bcic}.
	See also the proposition below for more details.
\end{rem}

\begin{prop}\label{prop:comp}
	Let $ \lambda \geq 0 $ be a constant and let $ \gamma: \R \to \R $ be a continuous function.
	Let $ \sigma \in L^\infty(\Omega \times (0,T)) \cap W^{1,2}(0,T;L^p(\Omega)) $, where $ p = \max \{ 1, N/2 \} $ if $ N \neq 2 $ and $ p>1 $ if $ N=2 $.
	Moreover, let $ z_0^j \in X \cap V $ and $ f^j \in L^2(0,T;L^2(\Omega)) \cap W^{1,2}(0,T;V^*) $ for $ j=1,2 $, and suppose that all these functions comply with {\rm (i)--(v)} of Theorem {\rm \ref{thm:main}} for each $ j $.
	Let $ z^j $ be a solution to \eqref{eq} and \eqref{bcic} with $ z_0 = z_0^j $ and $ f = f^j $ for $ j = 1,2 $ obtained via Theorem {\rm \ref{thm:main}}.
	Then it holds that
	\[
		z^1 \leq z^2 \quad \mbox{ a.e.~in } \ \Omega \times (0,T),
	\]
	provided that
	\[
		z_0^1 \leq z_0^2 \quad \mbox{ a.e.~in } \ \Omega \quad \mbox{ and } \quad f^1 \leq f^2 \quad \mbox{ a.e.~in } \ \Omega \times (0,T).
	\]
\end{prop}

\begin{proof}
	From the proof of Theorem \ref{thm:main}, there exists a strong solution $z^j=z^j(x,t)$ to \eqref{eq} and \eqref{bcic} with $ f = f^j $ and $ z_0 = z^j_0 $ for each $ j=1,2 $ which is obtained as a limit in $C([0,T];V)$ of the piecewise linear interpolant $ z^j_\tau $ of $ \{z^j_k\}_{k=0,\ldots,m} $, where $ z^j_k $ is a solution to the corresponding discretized problem \eqref{eq:dis} and \eqref{bcic:dis} for each $ k = 1,\dots,m $ and $ j = 1,2 $ (see \eqref{def:ztau} for the definition of $ \ztau $).
	Let $ \{ f^j_k \}_{k=1,\ldots,m} $ be the functions defined by \eqref{def:fk2} for each $ j $.
	Since one can check from the assumption that $f^1_k \leq f^2_k$ for $k=1,2,\dots,m$, we may iteratively deduce from Theorem \ref{thm:dis} that
	\[
		z^1_k \leq z^2_k \quad \mbox{ a.e.~in } \ \Omega \quad \mbox{ for } \ k=1,\dots,m;
	\]
	indeed, for the $k$-th step of the iteration, all the assumptions of Theorem \ref{thm:dis} are satisfied provided that $ z^1_{k-1} \leq z^2_{k-1} $ a.e.~in $ \Omega $.
	Moreover, by assumption, it also follows that $ z^1_0 \leq z^2_0 $ a.e.~in $ \Omega $.
	Thus we infer that
	\begin{equation}\label{eq:tau}
		z^1_\tau(\cdot,t) \leq z^2_\tau(\cdot,t) \quad \mbox{ a.e.~in } \ \Omega
	\end{equation}
	for all $ t \in [0,T] $ and $ \tau \in (0,1) $.
	Therefore we obtain the assertion by passing to the limit of \eqref{eq:tau} as $\tau \to 0_+$ recalling that $z^j_\tau \to z^j$ strongly in $C([0,T];V)$ for $j=1,2$.
\end{proof}

\section{Qualitative properties of solutions}\label{sec:qua}

The aim of this section is to prove Theorem \ref{thm:qua} and discuss the qualitative properties of strong solutions to \eqref{eq} and \eqref{bcic}.
Let $ \E : V \times [0,T] \to \R $ be the energy functional defined by \eqref{eq:energy}.

\begin{proof}[Proof of Theorem {\rm \ref{thm:qua}}]
	Let $z = z(x,t)$ be a strong solution to \eqref{eq} (or equivalently \eqref{vi}) and \eqref{bcic}.
	The irreversibility (i) follows immediately from \eqref{vi}.
	Hence we shall prove (ii) and (iii) below.
	Set $ \beta(s) = \gamma(s) + Ls $ for $ s \in \R $, where $ L>0 $ is a constant satisfying \eqref{ass:gammal} and \eqref{ass:gammau}.
	Then for a.e.~$t \in [0,T]$ and any $v \in V$ satisfying $v \leq z(t)$ a.e.~in $\Omega$, we find that
	\begin{align*}
		0 &\leq \int_\Omega \left( -\Delta z(t) + \lambda z(t) + \sigma(t) \gamma(z(t)) - f(t) \right) \left( v-z(t) \right) \d x\\
		&= \int_\Omega \left\{ -\Delta z(t) + (\lambda - L\sigma(t))z(t) + \sigma(t) \beta ( z(t) ) - f(t) \right\} \left( v-z(t) \right) \d x\\
		&\leq \frac{1}{2} \left\| \nabla v \right\|^2_{L^2(\Omega)} - \frac{1}{2} \left\| \nabla z(t) \right\|^2_{L^2(\Omega)} + \frac{1}{2} \int_\Omega \left( \lambda - L\sigma(t) \right) \left| v \right|^2 \d x - \frac{1}{2} \int_\Omega \left( \lambda - L\sigma(t) \right) \left| z(t) \right|^2 \d x\\
		&\quad + \int_\Omega \sigma(t) (\hat{\beta}(v) - \hat{\beta}(z(t))) \, \d x - \left\langle f(t), v \right\rangle_V + \left\langle f(t), z(t) \right\rangle_V\\
		&= \E (v,t) - \E (z(t),t).
	\end{align*}
	Furthermore, since $ \sigma \hat{\gamma}(z) \in W^{1,1}(0,T;L^1(\Omega)) $, the map $ t \mapsto \int_\Omega \sigma(t) \hat{\gamma}(z(t)) \, \d x $ is absolutely continuous on $ [0,T] $.
	Therefore $ t \mapsto \E(z(t),t) $ is absolutely continuous on $ [0,T] $ (see Subsection \ref{ssec:chain}), and thus we have 
	\[
		\E ( z(t),t ) \leq \E ( v,t ) \quad \mbox{ for any } \ t \in [0,T],
	\]
	provided that $v \in V$ and $v \leq z(t)$ a.e.~in $\Omega$.
	Indeed, let $ t_0 \in [0,T] $ be fixed and let $ v \in V $ be such that $ v \leq z(t_0) $ a.e.~in $ \Omega $. One can take a subsequence $ (t_n) $ in $ (0,t_0) $ such that the above relation holds at each $ t=t_n $ and $ t_n \uparrow t_0 $. Then noting by (i) and $ t_n < t_0 $ that $ v \leq z(t_0) \leq z(t_n) $ a.e.~in $ \Omega $, we infer that
	\[
		\E(z(t_n),t_n) \leq \E(v,t_n) \quad \mbox{ for all } \ n \in \N.
	\]
	Hence employing the fact that $ z \in C([0,T];V) $ and $ f \in W^{1,2}(0,T;V^*) \subset C([0,T];V^*) $ and passing to the limit as $ n \to \infty $, we obtain the desired relation at $ t=t_0 $.
	Thus (ii) follows.

	For the absolute continuity of the function $ t \mapsto \E(z(t),t) $, see Subsection \ref{ssec:chain} and Appendix \ref{appendix:ac}.
	Using \eqref{eq:ac}, we have
	\begin{align*}
		\frac{\d}{\d t} \, \E ( z(t),t ) &= \int_\Omega \left( -\Delta z(t) + \lambda z(t) + \sigma(t) \gamma(z(t)) -f(t) \right) \partial_t z(t) \,\d x\\
		&\quad + \int_\Omega \partial_t \sigma(t) \hat{\gamma}(z(t)) \, \d x - \left\langle \partial_t f(t), z(t) \right\rangle_V\\
		&= \int_\Omega \partial_t \sigma(t) \hat{\gamma}(z(t)) \, \d x - \left\langle \partial_t f(t), z(t) \right\rangle_V 
	\end{align*}
	for a.e.~$t \in (0,T)$. Here we also used (iii) of Definition \ref{defi:sol}. Hence it follows that
	\[
		\E ( z(t),t ) - \E ( z(s),s ) = \int^t_s \int_\Omega \partial_t \sigma(r) \hat{\gamma}(z(r)) \, \d x \, \d r - \int^t_s \left\langle \partial_t f(r), z(r) \right\rangle_V \d r
	\]
	for all $t,s \in [0,T]$. The proof is completed.
\end{proof}

\section{Global-in-time solutions}\label{sec:global}

This section is concerned with global-in-time solutions to the initial-boundary value problems \eqref{eq} (or \eqref{vi}) and \eqref{bcic}.
We shall first prove Theorem \ref{thm:global} to show the existence of global-in-time solutions for the initial-boundary problem.

\begin{proof}[Proof of Theorem {\rm \ref{thm:global}}]
	Fix $T>1$.
	Thanks to Theorem \ref{thm:main}, there exists a strong solution $ Z_1 \in W^{1,2}(0,T;V) \cap L^2(0,T;X) $ to \eqref{eq} and \eqref{bcic} on $ [0,T] $ with $ f $ and $ \sigma $ replaced by $ \left. f \right|_{(0,T)} $ and $ \left. \sigma \right|_{(0,T)} $, respectively.
	Let $ T_1 \in (1,T) $ be such that $ Z_1(T_1) \in X \cap V $.
	As in the previous step, one can obtain a strong solution $ Z_2 \in W^{1,2}(T_1,T_1+T;V) \cap L^2(T_1,T_1+T;X) $ to \eqref{eq} and \eqref{bcic} on $[T_1,T_1+T]$ with $ f $ and $ \sigma $ replaced by $ f = \left. f \right|_{(T_1,T_1+T)} $ and $ \sigma = \left. \sigma \right|_{(T_1,T_1+T)} $, respectively, with initial value $ Z_2(T_1) = Z_1(T_1) $.
	Here, we note that $ T>0 $ can be chosen arbitrary in Theorem \ref{thm:main}.
	Let $ T_2 \in (T_1+1,T_1+T) $ be such that $ Z_2(T_2) \in X \cap V $.
	By induction, one can define $ T_j \in (T_{j-1}+1, T_{j-1}+T) $ and $ Z_j \in W^{1,2}(T_{j-1},T_j;V) \cap L^2(T_{j-1},T_j;X) $ for $ j=1,2,\dots $, where $ Z_j $ is a strong solution to \eqref{eq} and \eqref{bcic} on $ [T_{j-1},T_j] $ satisfying that $ Z_j(T_{j-1}) = Z_{j-1}(T_{j-1}) $ a.e.~in $ \Omega $ for each $ j $ (here we set $ T_0 := 0 $ and $ Z_0(0) := z_0 $).
	In particular, it follows that $ T_j > j $ for $ j = 1,2,\dots $, since $ T_j > T_{j-1} + 1 $ and $ T_1 > 1 $.
	Then a function $ z: \Omega \times (0,\infty) \to \R $ defined by $ z(t) := Z_j(t) $ for $ t \in [T_{j-1},T_j] $ and $ j \in \N $, turns out to be a global-in-time strong solution to \eqref{eq} and \eqref{bcic}, and moreover, it satisfies $ z \in W^{1,2}_\loc \left( [0,\infty);V \right) \cap L^2_\loc \left( [0,\infty);X \right) $.
\end{proof}

We shall next give a proof of Theorem \ref{thm:asympt} concerning the long-time behavior of global-in-time strong solutions.
Before that, let us give a couple of remarks on the theorem:

\begin{rem}\label{rem:global}
	\begin{enumerate}
		\item To prove the convergence of global-in-time solutions of \eqref{eq} and \eqref{bcic} to a stationary limit, we need to consider a solution which is constructed via the minimizing movement scheme as in Section \ref{sec:main}; indeed, the key estimates in Lemma \ref{lem:est} are derived from some estimates for discretized solutions $ \{ z_k \} $ (see \eqref{est:zk}, \eqref{est:zkd} and \eqref{est:zkl}).
		\item In the case where $ \sigma $ depends on both $ x $ and $ t $, Theorem \ref{thm:asympt} holds true if the assumptions (vii) and \eqref{ass:fglobal} are replaced by the following: suppose that there exists a function $ \sigma_\infty \in L^\infty(\Omega) $ such that $ \sigma(t) \to \sigma_\infty $ strongly in $ L^\infty(\Omega) $, and it holds that
		\begin{equation}\label{eq:xi}
			f_\infty - \lambda \xi - \sigma_\infty \gamma(\xi) \leq f(t) - \lambda z_0 - \sigma(t) \gamma(z_0) \quad \mbox{ a.e.~in } \ \Omega \times (0,\infty),
		\end{equation}
		where $ \xi \in X \cap V $ is a (unique) strong solution to \eqref{limvi} and \eqref{limbc}.
	\end{enumerate}
\end{rem}

The following lemma gives a priori estimates of global-in-time solution $ z $ in $ \R_+ = (0,\infty) $.

\begin{lem}\label{lem:est}
	Suppose that all the assumptions of Theorem {\rm \ref{thm:asympt}} are fulfilled and let $ z : \Omega \times (0,\infty) \to \R $ be a global-in-time strong solution to \eqref{eq} {\rm (}or equivalently \eqref{vi}{\rm )} and \eqref{bcic} obtained by Theorem {\rm \ref{thm:global}} {\rm (}via the time-discretization scheme exhibited in Section {\rm \ref{sec:main}}{\rm )}.
	Then there exists a constant $ C>0 $ independent of $ z $ and $ t $ such that
	\begin{align}
		&\left\| z(t) \right\|_V \leq C, \notag \\
		&\left\| \partial_t z(t) \right\|^2_V \leq C \left\| \partial_t f(t) \right\|^2_{V^*}, \notag \\
		&\left\| \Delta z(t) \right\|^2_{L^2(\Omega)} \leq C \left( \| f_\infty \|^2_{L^2(\Omega)} + \left\| z_0 \right\|^2_X + 1 \right) \label{eq:j3}
	\end{align}
	for a.e.~$ t \in \R_+ $, where $ f_\infty $ is the limit function given in {\rm (viii)} of Theorem {\rm \ref{thm:asympt}}.
	In particular, it holds that $ z \in L^\infty(0,\infty;V) $ and $ \partial_t z \in L^2(0,\infty;V) $.
\end{lem}

\begin{proof}
	Let $ T>0 $, $ m \in \N $ and $ \tau = T/m > 0 $.
	Set $ t_k = k \tau $ for $ k=1,\dots,m $ and define $ f_k $ for $ k=0,1,\dots,m $ by \eqref{def:fk2} and \eqref{def:f02}.
	Let $ z_k \in X \cap V $ be a strong solution to \eqref{eq:dis} and \eqref{bcic:dis} for $ k=1,\dots,m $.
	According to Lemmas \ref{lem:zk}--\ref{lem:zkl}, the interpolants $ (\ztau) $ and $ (\ztaubar) $ of $ \{ z_k \} $, which are defined by \eqref{def:ztau} and \eqref{def:ztaubar}, satisfy
	\begin{align*}
		&\left\| \ztaubar(t) \right\|_V \leq \phi_{f,T},\\
		&\left\| \partial_t \ztau(t) \right\|^2_V \leq C \phi_{f,T}^2 \left\| \partial_t \ftau(t) \right\|^2_{V^*},\\
		&\left\| \Delta \ztaubar(t) \right\|^2_{L^2(\Omega)} \leq C \left( \left\| \ztaubar(t) \right\|^2_V + \left\| \ftaubar(t) \right\|^2_{L^2(\Omega)} + \| f_\infty \|^2_{L^2(\Omega)} + \left\| z_0 \right\|^2_X + 1 \right)
	\end{align*}
	for a.e.~$ t \in (0,T) $, where
	\[
		\phi_{f,T} := \phi \left( \left\| \partial_t f \right\|_{L^1(0,T;V^*)} \right)
	\]
	(see \eqref{eq:defphi} for definition of $ \phi $), and $ C>0 $ is a constant independent of $ T $.
	Moreover, $ \phi_{f,T} $ is bounded with respect to $ T $ since $ \partial_t f \in L^1(0,\infty;V^*) $.
	Hence, by virtue of Proposition \ref{prop:inter}, we can extract a (not relabeled) subsequence of $ (\tau) $ such that
	\begin{alignat*}{4}
		\ztau &\to z & \quad &\mbox{ weakly in } \ W^{1,2}(0,T;V),\\
		&&&\mbox{ strongly in } \ C([0,T];V),\\
		\ztaubar &\to z &&\mbox{ weakly in } \ L^2(0,T;X),\\
		&&&\mbox{ strongly in } \ L^\infty(0,T;V).
	\end{alignat*}
	Therefore, combined with the lower-semicontinuity of the $ L^2 $-norm, it follows that
	\begin{align*}
		\left\| z(t) \right\|_V &\leq C,\\
		\left\| \partial_t z(t) \right\|^2_V &\leq C \left\| \partial_t f(t) \right\|^2_{V^*},\\
		\left\| \Delta z(t) \right\|^2_{L^2(\Omega)} &\leq C \left( \left\| f(t) \right\|^2_{L^2(\Omega)} + \| f_\infty \|^2_{L^2(\Omega)} + \left\| z_0 \right\|^2_X + 1 \right)\\
		&\leq C \left( \left\| f_\infty \right\|^2_{L^2(\Omega)} + \left\| z_0 \right\|^2_X + 1 \right)
	\end{align*}
	for a.e.~$ t \in (0,T) $, where $ C>0 $ is independent of $ T $ (see also Propositions \ref{prop:inter} and \ref{pro:f} in Appendix \ref{appendix:inter}).
	From the arbitrariness of $ T $, the lemma follows.
\end{proof}

We are now in a position to prove Theorem \ref{thm:asympt}.

\begin{proof}[Proof of Theorem {\rm \ref{thm:asympt}}]
	Let $ z $ be a global-in-time strong solution to \eqref{eq} and \eqref{bcic} constructed via the time-discretization scheme (see Theorem \ref{thm:global} and Section \ref{sec:main}).
	Since the family $(-\Delta z(t))_{t \in \R_+}$ is bounded in $ L^2(\Omega) $ by virtue of \eqref{eq:j3}, there exist a sequence $ (s_n) \subset \R_+ $ and $z_\infty \in X \cap V$ such that $ s_1 < s_2 < \dots < s_n \to +\infty $ and
	\begin{alignat}{4}
		z(s_n) &\to z_\infty & \quad &\mbox{ strongly in } \ V, \label{eq:zsnst} \\
		\Delta z(s_n) &\to \Delta z_\infty &&\mbox{ weakly in } \ L^2(\Omega). \label{eq:zsnw}
	\end{alignat}
	Furthermore, since $ z $ is nonincreasing in $ t $, it follows that $ z(t) \to z_\infty $ strongly in $ V $; indeed, let $ (r_n) $ be a sequence in $ \R_+ $ which satisfies that $ r_1 < r_2 < \dots < r_n\to +\infty $.
	By a compactness argument similar as above, one can take a (not relabeled) subsequence of $ (n) $ and $ \tilde{z}_\infty \in X \cap V $ such that $ z(r_n) \to \tilde{z}_\infty $ strongly in $ V $.
	However, since $ (s_n) $ grows monotonically, we can extract a further (not relabeled) subsequence of $ (n) $ such that $ r_n < s_n $ for all $ n \in \N $.
	Then it follows from the nonincrease of $ z $ in $ t $ that $ z(s_n) \leq z(r_n) $ a.e.~in $ \Omega $, and thus $ z_\infty \leq \tilde{z}_\infty $ a.e.~in $ \Omega $.
	On the other hand, we can also prove that $ z_\infty \geq \tilde{z}_\infty $ a.e.~in $ \Omega $ by a similar argument.
	Therefore $ z_\infty = \tilde{z}_\infty $ a.e.~in $ \Omega $, and in particular, we can deduce from the convergences \eqref{eq:zsnst} and \eqref{eq:zsnw} that
	\begin{equation}\label{eq:zinfvi}
		z_\infty \leq z_0 \quad \mbox{ and } \quad {-\Delta} z_\infty + \lambda z_\infty + \tilde{\sigma} \gamma(z_\infty) \leq f_\infty \quad \mbox{ a.e.~in } \ \Omega,
	\end{equation}
	where we note that $ \gamma(z(t)) \to \gamma(z_\infty) $ strongly in $ L^q(\Omega) $ by virtue of \eqref{ass:gammau}.
	
	Finally we prove that $ z_\infty $ complies with \eqref{limvi} and \eqref{limbc}.
	Let $ \xi \in X \cap V $ be the (unique) strong solution to \eqref{limvi} and \eqref{limbc}, where Theorem \ref{thm:dis} can also be applied to prove the existence of such $ \xi $ with only obvious modifications.
	We shall show that $ z_\infty = \xi $, provided that \eqref{ass:fglobal} is satisfied.
	Indeed, $ \Xi(t) := \xi $ turns out to be a strong solution to \eqref{eq} and \eqref{bcic} with $ \sigma = \tilde{\sigma} $, $ f(t) = f_\infty $ and $ z_0 = \xi $.
	Applying the comparison principle for \eqref{eq} and \eqref{bcic} (see Proposition \ref{prop:comp}) and using \eqref{ass:fglobal} along with $ z(t) \leq z_0 $, we infer that
	\[
		\xi(x) = \Xi(x,t) \leq z(x,t) \quad \mbox{ for a.e.~} (x,t) \in \Omega \times \R_+. 
	\]
	Hence passing to the limit as $t \to +\infty$, we obtain
	\[
		\xi \leq z_\infty \quad \mbox{ a.e.~in } \ \Omega.
	\]
	On the other hand, it is obvious that $ w = z_\infty $ is a solution to
	\begin{equation}\label{eq:psig}
		\left\lbrace
		\begin{alignedat}{4}
			&\dind(w-\psi) - \Delta w + \lambda w + \tilde{\sigma} \gamma(w) \ni g \quad &&\mbox{ in } \ \Omega,\\
			&w = 0 &&\mbox{ on } \ \ddom,\\
			&\partial_\nu w = 0 &&\mbox{ on } \ \ndom
		\end{alignedat}
		\right.
	\end{equation}
	with $ \psi = z_\infty $ and $ g = -\Delta z_\infty + \lambda z_\infty + \tilde{\sigma} \gamma(z_\infty) $.
	Moreover, one can check that \eqref{eq:psig} enjoys the comparison principle with respect to $ \psi $ and $ g $ similarly to the proof of Theorem \ref{thm:dis}.
	Since $ \xi $ is a solution to \eqref{eq:psig} with $ \psi = z_0 $ and $ g = f_\infty $, it follows from \eqref{eq:zinfvi} that
	\[
		z_\infty \leq \xi \quad \mbox{ a.e.~in } \ \Omega.
	\]
	Thus we obtain $ z_\infty = \xi $, which completes the proof.
\end{proof}

\section*{Acknowledgments}

This is a part of the author's Ph.D.~thesis.
The author would like to express his sincere gratitude to his supervisor Professor Goro Akagi for his valuable advice and constant help.
The author is supported by Grant-in-Aid for JSPS fellows (No.~JP21J20732).
This work was also supported by the Research Institute for Mathematical Sciences, an International Joint Usage/Research Center located in Kyoto University.

\appendix

\section{Estimate for interpolation functions}\label{appendix:inter}

In this section, we show the boundedness from above of interpolant functions.
Let $ B $ be a general Banach space with norm $ \| \cdot \|_B $ and let $m \in \N$, $ T>0 $, $\tau := T/m$ and $ t_k := k \tau $ for $ k=1,\dots,m $.
Let $ 1 \leq p < \infty $.
For $ f \in L^p(0,T;B) $, set
\begin{align*}
	f_k := \frac{1}{\tau} \int^{t_k}_{t_{k-1}} f(s) \, \d s \quad \mbox{ for } \ k=1, \dots, m,
\end{align*}
and moreover, if $ f \in C([0,T];B) $, set $ f_0 := f(0) $.
Define piecewise linear and constant interpolants of $ \{ f_k \} $ by
\begin{align*}
	\ftau(t) &:= f_{k-1} + \frac{t - t_{k-1}}{\tau} \left( f_k-f_{k-1} \right),\\
	\ftaubar(t) &:= f_k
\end{align*}
for $t_{k-1} < t \leq t_k$, $ k=1,\dots,m $, respectively.
Furthermore, we set $ \ftau(\cdot,0) = f_0 $.

\begin{prop}\label{prop:inter}
	Assume that $ f \in W^{1,p}(0,T;B) $ {\rm (}$ \subset C([0,T];B) ${\rm )}.
	For $ m \in \N $ and $ p \geq 1 $, it holds that
	\begin{align*}
		\tau \sum^m_{k=1} \left\| \frac{f_k-f_{k-1}}{\tau} \right\|^p_B \leq 2^p \int^T_0 \left\| \partial_t f(t) \right\|^p_B \d t.
	\end{align*}
\end{prop}

\begin{proof}
	For $ 2 \leq k \leq m $, 
	\begin{align*}
	 	\left\| \frac{f_k-f_{k-1}}{\tau} \right\|_B &= \frac{1}{\tau^2} \left\| \int^{t_k}_{t_{k-1}} f(s) \, \d s - \int^{t_{k-1}}_{t_{k-2}} f(s) \, \d s \right\|_B \\
		&= \frac{1}{\tau^2} \left\| \int^{t_k}_{t_{k-1}} \left( f(s) - f(s-\tau) \right) \d s \right\|_B \\
		&\leq \frac{1}{\tau^2} \int^{t_k}_{t_{k-1}} \int^s_{s-\tau} \left\| \partial_t f(r) \right\|_B \d r \, \d s \leq \frac{1}{\tau} \int^{t_k}_{t_{k-2}} \left\| \partial_t f(r) \right\|_B \d r.
	\end{align*}
	Similarly, we have
	\[
	 	\left\| \frac{f_1 - f_0}{\tau} \right\|_B \leq \frac{1}{\tau} \int^\tau_0 \left\| \partial_t f(r) \right\|_B \d r.
	\]
	Therefore one can deduce that
	\begin{align*}
		\tau \sum^m_{k=1} \left\| \frac{f_k-f_{k-1}}{\tau} \right\|^p_B &\leq \tau \left( \frac{1}{\tau} \int^\tau_0 \left\| \partial_t f(r) \right\|_B \d r \right)^p + \tau \sum^m_{k=2} \left( \frac{1}{\tau} \int^{t_k}_{t_{k-2}} \left\| \partial_t f(r) \right\|_B \d r \right)^p\\
		&\leq \int^\tau_0 \left\| \partial_t f(r) \right\|^p_B \d r + 2^{p-1} \sum^m_{k=2} \int^{t_k}_{t_{k-2}} \left\| \partial_t f(r) \right\|^p_B \d r\\
		&\leq 2^p \int^T_0 \left\| \partial_t f(r) \right\|^p_B \d r,
	\end{align*}
	which completes the proof.
\end{proof}

Furthermore, we have

\begin{prop}\label{pro:f}
	It holds that
	\begin{alignat*}{4}
		\ftaubar &\to f & \quad &\mbox{ strongly in } \ L^p(0,T;B)
	\end{alignat*}
	as $\tau \to 0_+$.
\end{prop}

For a proof of Proposition \ref{pro:f}, see \cite{e2ciec}*{Appendix A}, where only the case $ p=2 $ and $ B = L^2(\Omega) $ is explicitly treated.
However, the proof can also be applied to $ L^p(0,T;B) $ with $ 1 \leq p < \infty $ and a general Banach space $ B $ with only obvious modifications.

\section{Absolute continuity of the energy functional in time}\label{appendix:ac}

As in \eqref{eq:energy}, we define
\[
	\E(u,t) := \frac{1}{2} \left\| \nabla u \right\|^2_{L^2(\Omega)} + \frac{\lambda}{2} \left\| u \right\|^2_{L^2(\Omega)} + \int_\Omega \sigma(t) \hat{\gamma}(u) \, \d x - \left\langle f(t), u \right\rangle_V
\]
for $ u \in V $ and $ t \in [0,T] $.
Let $ z \in W^{1,2}(0,T;V) \cap L^2(0,T;X) $.
To complete the proof of Theorem \ref{thm:qua}, we show that the function $ t \mapsto \E(z(t),t) $ is absolutely continuous on $ [0,T] $.

\begin{lem}\label{lem:ac}
	Let $ f $, $ \sigma $ and $ \gamma $ satisfy all the assumptions in Theorem {\rm \ref{thm:qua}} and let $ z \in W^{1,2}(0,T;V) \cap L^2(0,T;X) $.
	Then the function
	\[
		[0,T] \ni t \mapsto \E(z(t),t)
	\]
	is absolutely continuous, and moreover, it holds that
	\begin{align}
		\frac{\d}{\d t} \, \E(z(t),t) &= \left( -\Delta z(t) + \lambda z(t), \partial_t z(t) \right)_{L^2(\Omega)} + \left( \partial_t \sigma(t), \hat{\gamma}(z(t)) \right)_{L^2(\Omega)} \notag \\
		&\quad + \left( \sigma(t) \gamma(z(t)), \partial_t z(t) \right)_{L^2(\Omega)} - \left\langle \partial_t f(t), z(t) \right\rangle_V - \left( f(t), \partial_t z(t) \right)_{L^2(\Omega)}, \label{eq:ac}
	\end{align}
	where $ (\cdot, \cdot)_{L^2(\Omega)} $ denotes the inner product of $ L^2(\Omega) $, for a.e.~$ t \in (0,T) $.
\end{lem}

\begin{proof}
	It is already known that the first term and the second term of $ \E(z(t),t) $ are absolutely continuous on $ [0,T] $ (see Subsection \ref{ssec:chain}).
	As for the third term, we set
	\[
		\Phi(t) := \int_\Omega \sigma(t) \hat{\gamma}(z(t)) \, \d x
	\]
	for $ t \in [0,T] $.
	According to \eqref{ass:gammal}, we can assume that $ \hat{\gamma} $ is convex without loss of generality.
	Then it follows that
	\begin{align*}
		\left| \Phi(t) - \Phi(s) \right| &\leq \int_\Omega \left| \sigma(t) \right| \left| \hat{\gamma}(z(t)) - \hat{\gamma}(z(s)) \right| \d x + \int_\Omega \left| \sigma(t) - \sigma(s) \right| \left| \hat{\gamma}(z(s)) \right| \d x\\
		&\leq \left\| \sigma \right\|_{L^\infty(\Omega \times (0,T))} \left( \left\| \gamma(z(t)) \right\|_{L^2(\Omega)} + \left\| \gamma(z(s)) \right\|_{L^2(\Omega)} \right) \left( \left\| z(t) - z(s) \right\|_{L^2(\Omega)} \right)\\
		&\quad + \left\| \hat{\gamma}(z(s)) \right\|_{L^{q/2}(\Omega)} \int^t_s \left\| \partial_t \sigma(r) \right\|_{L^p(\Omega)} \d r\\
		&\leq C \left\| \sigma \right\|_{L^\infty(\Omega \times (0,T))} \left( \left\| z(t) \right\|_{L^2(\Omega)} + \left\| z(s) \right\|_{L^2(\Omega)} + 1 \right) \int^t_s \left\| \partial_t z(r) \right\|_{L^2(\Omega)} \d r\\
		&\quad + C \left( \left\| z(s) \right\|^2_{L^q(\Omega)} + 1 \right) \int^t_s \left\| \partial_t \sigma(r) \right\|_{L^p(\Omega)} \d r
	\end{align*}
	for a.e.~$ t,s \in (0,T) $.
	Since $ (z(t))_{t \in [0,T]} $ is bounded in $ V $, we have
	\[
		\left| \Phi(t) - \Phi(s) \right| \leq C \int^t_s \left( \left\| \partial_t z(r) \right\|_{L^2(\Omega)} + \left\| \partial_t \sigma(r) \right\|_{L^p(\Omega)} \right) \d r,
	\]
	where $ C>0 $ is a constant independent of $ t $ and $ s $.
	Therefore $ t \mapsto \Phi(t) $ is absolutely continuous on $ [0,T] $.

	Moreover, let $ t \in (0,T) $ and let $ h \in \R $ satisfy $ t+h \in (0,T) $.
	One can deduce that
	\begin{align*}
		\frac{\sigma(t+h) \hat{\gamma}(z(t+h)) - \sigma(t) \hat{\gamma}(z(t))}{h} \to \partial_t \sigma(t) \hat{\gamma}(z(t)) + \sigma(t) \gamma(z(t)) \partial_t z(t) \quad \mbox{ as } \ h \to 0
	\end{align*}
	a.e.~in $ \Omega $.
	Furthermore, it follows that
	\begin{align*}
		\mel{\left| \frac{\sigma(t+h) \hat{\gamma}(z(t+h)) - \sigma(t) \hat{\gamma}(z(t))}{h} \right|}\\
		&\leq \left\| \sigma \right\|_{L^\infty(\Omega \times (0,T))} \left| \frac{\hat{\gamma}(z(t+h))-\hat{\gamma}(z(t))}{h} \right| + \left| \hat{\gamma}(z(t)) \right| \left| \frac{\sigma(t+h) - \sigma(t)}{h} \right|\\
		&\leq \left\| \sigma \right\|_{L^\infty(\Omega \times (0,T))} \left| \frac{1}{h} \int^{t+h}_t \gamma(z(r)) \partial_t z(r) \, \d r \right| + \left| \hat{\gamma}(z(t)) \right| \left| \frac{1}{h} \int^{t+h}_t \partial_t \sigma(r) \, \d r \right|\\
		&\leq \left\| \sigma \right\|_{L^\infty(\Omega \times (0,T))} \left( \left| \gamma(z(t)) \right| \left| \partial_t z(t) \right| + 1 \right) + \left| \hat{\gamma}(z(t)) \right| \left( \left| \partial_t \sigma(t) \right| + 1 \right)
	\end{align*}
	a.e.~in $ \Omega $.
	Therefore, by the use of the dominated convergence theorem, we obtain
	\[
		\Phi'(t) = \int_\Omega \left( \partial_t \sigma(t) \hat{\gamma}(z(t)) + \sigma(t) \hat{\gamma}(z(t)) \partial_t z(t) \right) \d x
	\]
	for a.e.~$ t \in (0,T) $.
	Finally, as for the fourth term of $ \E(z(t),t) $, see \cite{e2ciec}*{Proposition B.1}.
\end{proof}




\begin{bibdiv}
\begin{biblist}

\bib{ak}{article}{
      author={Akagi, Goro},
      author={Kimura, Masato},
       title={\emph{Unidirectional evolution equations of diffusion type}},
        date={2019},
        ISSN={0022-0396},
     journal={J. Differential Equations},
      volume={266},
      number={1},
       pages={1\ndash 43},
         url={https://doi.org/10.1016/j.jde.2018.05.022},
}

\bib{e2ciec}{article}{
      author={Akagi, Goro},
      author={Sato, Kotaro},
       title={\emph{Evolution equations with complete irreversibility and
  energy conservation}},
        date={2023},
        ISSN={0022-247X},
     journal={J. Math. Anal. Appl.},
      volume={527},
      number={1},
       pages={Paper No. 127348},
         url={https://doi.org/10.1016/j.jmaa.2023.127348},
}

\bib{at1990}{article}{
      author={Ambrosio, Luigi},
      author={Tortorelli, Vincenzo~Maria},
       title={\emph{Approximation of functionals depending on jumps by elliptic
  functionals via {$\Gamma$}-convergence}},
        date={1990},
     journal={Comm. Pure Appl. Math.},
      volume={43},
      number={8},
       pages={999\ndash 1036},
         url={https://doi.org/10.1002/cpa.3160430805},
}

\bib{at1992}{article}{
      author={Ambrosio, Luigi},
      author={Tortorelli, Vincenzo~Maria},
       title={\emph{On the approximation of free discontinuity problems}},
        date={1992},
     journal={Boll. Un. Mat. Ital. B (7)},
      volume={6},
      number={1},
       pages={105\ndash 123},
}

\bib{arai}{article}{
      author={Arai, Tsutomu},
       title={\emph{On the existence of the solution for $\partial
  \varphi(u'(t)) + \partial \psi(u(t)) \ni f(t)$}},
        date={1979},
     journal={J. Fac. Sci. Univ. Tokyo sec. IA Math.},
      volume={26},
       pages={75\ndash 96},
}

\bib{barbu1975}{article}{
      author={Barbu, Viorel},
       title={\emph{Existence theorems for a class of two point boundary
  problems}},
        date={1975},
        ISSN={0022-0396},
     journal={J. Differential Equations},
      volume={17},
       pages={236\ndash 257},
         url={https://doi.org/10.1016/0022-0396(75)90043-1},
}

\bib{barbu1976}{book}{
      author={Barbu, Viorel},
       title={Nonlinear {S}emigroups and {D}ifferential {E}quations in {B}anach
  {S}paces},
   publisher={Noordhoff International Publishing, Leiden},
        date={1976},
}

\bib{brezis1971}{inproceedings}{
      author={Br\'{e}zis, Haim},
       title={\emph{Monotonicity methods in {H}ilbert spaces and some
  applications to nonlinear partial differential equations}},
        date={1971},
   booktitle={Contributions to nonlinear functional analysis ({P}roc.
  {S}ympos., {M}ath. {R}es. {C}enter, {U}niv. {W}isconsin, {M}adison, {W}is.,
  1971)},
   publisher={Academic Press, New York},
       pages={101\ndash 156},
}

\bib{brezisF}{book}{
      author={Br\'{e}zis, Haim},
       title={Op\'{e}rateurs maximaux monotones et semi-groupes de contractions
  dans les espaces de {H}ilbert},
      series={North-Holland Mathematics Studies, No. 5},
   publisher={North-Holland Publishing Co., Amsterdam-London; American Elsevier
  Publishing Co., Inc., New York},
        date={1973},
}

\bib{ch}{book}{
      author={Cazenave, Thierry},
      author={Haraux, Alain},
       title={An introduction to semilinear evolution equations},
      series={Oxford Lecture Series in Mathematics and its Applications},
   publisher={The Clarendon Press, Oxford University Press, New York},
        date={1998},
      volume={13},
        ISBN={0-19-850277-X},
}

\bib{chambolle2003}{article}{
      author={Chambolle, Antonin},
       title={\emph{A density result in two-dimensional linearized elasticity,
  and applications}},
        date={2003},
        ISSN={0003-9527},
     journal={Arch. Ration. Mech. Anal.},
      volume={167},
      number={3},
       pages={211\ndash 233},
         url={https://doi.org/10.1007/s00205-002-0240-7},
}

\bib{clark}{article}{
      author={Clark, {D.}S.},
       title={\emph{Short proof of a discrete {G}ronwall inequality}},
        date={1987},
        ISSN={0166-218X},
     journal={Discrete Appl. Math.},
      volume={16},
      number={3},
       pages={279\ndash 281},
         url={https://doi.org/10.1016/0166-218X(87)90064-3},
}

\bib{cv}{article}{
      author={Colli, P.},
      author={Visintin, A.},
       title={\emph{On a class of doubly nonlinear evolution equations}},
        date={1990},
        ISSN={0360-5302},
     journal={Comm. Partial Differential Equations},
      volume={15},
      number={5},
       pages={737\ndash 756},
         url={https://doi.org/10.1080/03605309908820706},
}

\bib{colli}{article}{
      author={Colli, Pierluigi},
       title={\emph{On some doubly nonlinear evolution equations in {B}anach
  spaces}},
        date={1992},
        ISSN={0916-7005},
     journal={Japan J. Indust. Appl. Math.},
      volume={9},
      number={2},
       pages={181\ndash 203},
         url={https://doi.org/10.1007/BF03167565},
}

\bib{corduneanu}{book}{
      author={Corduneanu, C.},
       title={Principles of differential and integral equations},
     edition={Second edition},
   publisher={Chelsea Publishing Co., Bronx, N.Y.},
        date={1977},
}

\bib{dft2005}{article}{
      author={Dal~Maso, Gianni},
      author={Francfort, {G.}A.},
      author={Toader, Rodica},
       title={\emph{Quasistatic crack growth in nonlinear elasticity}},
        date={2005},
        ISSN={0003-9527},
     journal={Arch. Ration. Mech. Anal.},
      volume={176},
      number={2},
       pages={165\ndash 225},
         url={https://doi.org/10.1007/s00205-004-0351-4},
}

\bib{dt2002}{article}{
      author={Dal~Maso, Gianni},
      author={Toader, Rodica},
       title={\emph{A model for the quasi-static growth of brittle fractures:
  existence and approximation results}},
        date={2002},
        ISSN={0003-9527},
     journal={Arch. Ration. Mech. Anal.},
      volume={162},
      number={2},
       pages={101\ndash 135},
         url={https://doi.org/10.1007/s002050100187},
}

\bib{fl}{article}{
      author={Francfort, {G.}A.},
      author={Larsen, {C.}L.},
       title={\emph{Existence and convergence for quasi-static evolution in
  brittle fracture}},
        date={2003},
        ISSN={0010-3640},
     journal={Comm. Pure Appl. Math.},
      volume={56},
      number={10},
       pages={1465\ndash 1500},
         url={https://doi.org/10.1002/cpa.3039},
}

\bib{fm}{article}{
      author={Francfort, {G.}A.},
      author={Marigo, Jean-Jacques},
       title={\emph{Revisiting brittle fracture as an energy minimization
  problem}},
        date={1998},
        ISSN={0022-5096},
     journal={J. Mech. Phys. Solids},
      volume={46},
      number={8},
       pages={1319\ndash 1342},
         url={https://doi.org/10.1016/S0022-5096(98)00034-9},
}

\bib{giacomini}{article}{
      author={Giacomini, Alessandro},
       title={\emph{Ambrosio-{T}ortorelli approximation of quasi-static
  evolution of brittle fractures}},
        date={2005},
        ISSN={0944-2669},
     journal={Calc. Var. Partial Differential Equations},
      volume={22},
      number={2},
       pages={129\ndash 172},
         url={https://doi.org/10.1007/s00526-004-0269-6},
}

\bib{griffith}{article}{
      author={Griffith, {A.}A.},
       title={\emph{The phenomena of rupture and flow in solids}},
        date={1921},
        ISSN={02643952},
     journal={Philosophical Transactions of the Royal Society of London. Series
  A, Containing Papers of a Mathematical or Physical Character},
      volume={221},
       pages={163\ndash 198},
         url={http://www.jstor.org/stable/91192},
}

\bib{grisvard}{book}{
      author={Grisvard, P.},
       title={Elliptic problems in nonsmooth domains},
      series={Monographs and Studies in Mathematics},
   publisher={Pitman (Advanced Publishing Program), Boston, MA},
        date={1985},
      volume={24},
        ISBN={0-273-08647-2},
}

\bib{kmz}{article}{
      author={Knees, Dorothee},
      author={Mielke, Alexander},
      author={Zanini, Chiara},
       title={\emph{On the inviscid limit of a model for crack propagation}},
        date={2008},
        ISSN={0218-2025},
     journal={Math. Models Methods Appl. Sci.},
      volume={18},
      number={9},
       pages={1529\ndash 1569},
         url={https://doi.org/10.1142/S0218202508003121},
}

\bib{lm}{book}{
      author={Lions, J.-L.},
      author={Magenes, E.},
       title={Non-homogeneous boundary value problems and applications. {V}ol.
  {I}},
      series={Die Grundlehren der mathematischen Wissenschaften, Band 181},
   publisher={Springer-Verlag, New York-Heidelberg},
        date={1972},
}

\bib{ms}{article}{
      author={Magenes, Enrico},
      author={Stampacchia, Guido},
       title={\emph{I problemi al contorno per le equazioni differenziali di
  tipo ellittico}},
        date={1958},
        ISSN={0391-173X},
     journal={Ann. Scuola Norm. Sup. Pisa Cl. Sci. (3)},
      volume={12},
       pages={247\ndash 358},
}

\bib{mz}{article}{
      author={Mielke, Alexander},
      author={Zelik, Sergey},
       title={\emph{On the vanishing-viscosity limit in parabolic systems with
  rate-independent dissipation terms}},
        date={2014},
     journal={Ann. Scuola Norm. Sup. Pisa Cl. Sci. (5)},
      volume={13},
       pages={67\ndash 135},
}

\bib{rockafellar}{article}{
      author={Rockafellar, {R.}T.},
       title={\emph{On the maximal monotonicity of subdifferential mappings}},
        date={1970},
        ISSN={0030-8730},
     journal={Pacific J. Math.},
      volume={33},
       pages={209\ndash 216},
         url={http://projecteuclid.org/euclid.pjm/1102977253},
}

\bib{sss}{article}{
      author={Schimperna, Giulio},
      author={Segatti, Antonio},
      author={Stefanelli, Ulisse},
       title={\emph{Well-posedness and long-time behavior for a class of doubly
  nonlinear equations}},
        date={2007},
        ISSN={1078-0947},
     journal={Discrete Contin. Dyn. Syst.},
      volume={18},
      number={1},
       pages={15\ndash 38},
         url={https://doi.org/10.3934/dcds.2007.18.15},
}

\bib{simon}{article}{
      author={Simon, Jacques},
       title={\emph{Compact sets in the space {$L^p(0,T;B)$}}},
        date={1987},
        ISSN={0003-4622},
     journal={Ann. Mat. Pura Appl. (4)},
      volume={146},
       pages={65\ndash 96},
         url={https://doi.org/10.1007/BF01762360},
}

\end{biblist}
\end{bibdiv}
\end{document}